\documentclass[11pt, reqno]{amsart}
\usepackage[utf8]{inputenc}	
\usepackage{amssymb,amscd}
\usepackage{multirow,booktabs}
\usepackage[table]{xcolor}
\usepackage{fullpage}
\usepackage{lastpage}
\usepackage{fancyhdr}
\usepackage{mathrsfs}
\usepackage{wrapfig}
\usepackage{graphicx}
\usepackage{setspace}
\usepackage{calc}
\usepackage{bm}
\usepackage{tikz}
\usetikzlibrary{arrows.meta, positioning}
\usepackage{cite}
\usepackage{enumerate}
\usepackage{subfigure}
\usepackage{multicol}
\usepackage[colorlinks,
            linkcolor=blue,
            citecolor=red,
            ]{hyperref}
\usepackage{cancel}
\usepackage[all]{xy}
\usepackage{geometry}
\usepackage{empheq}
\usepackage{framed}
\usepackage{dsfont}
\usepackage{xcolor}
\theoremstyle{plain}
\newtheorem{theorem}{Theorem}[section]
\newtheorem{lemma}[theorem]{Lemma}
\newtheorem{proposition}[theorem]{Proposition}
\newtheorem{corollary}[theorem]{Corollary}
\newtheorem{assumption}[theorem]{Assumption}
\newtheorem{definition}[theorem]{Definition}
\theoremstyle{remark}

\newtheorem{example}{Example}[section]
\newtheorem{remark}{Remark}[section]

\usepackage{indentfirst}
\begin{document}

\newcommand{\QQ}{\mathbb{Q}}
\newcommand{\RR}{\mathbb{R}}
\newcommand{\ZZ}{\mathbb{Z}}
\newcommand{\NN}{\mathbb{N}}
\newcommand{\CC}{\mathbb{C}}
\newcommand{\CalC}{\mathcal{C}}
\newcommand{\EE}{\mathbb{E}}
\newcommand{\Var}{\operatorname{Var}}
\newcommand{\PP}{\mathbb{P}}
\newcommand{\Rd}{\mathbb{R}^d}
\newcommand{\Rn}{\mathbb{R}^n}
\newcommand{\FF}{\mathcal{F}}
\newcommand{\GG}{\mathcal{G}}
\newcommand{\F}{\mathscr{F}}

\author[Yuval Peres]{Yuval Peres}
\address[Yuval Peres]{Beijing Institute of Mathematical Sciences and Applications}
\email{yperes@gmail.com}
  
\author[Shuo Qin]{Shuo Qin}
\address[Shuo Qin]{Beijing Institute of Mathematical Sciences and Applications, and Yau Mathematical Sciences Center, Tsinghua University}
\email{qinshuo@bimsa.cn}

\title{Mixing times of step-reinforced random walks}
\date{}

  \begin{abstract}
  We study the mixing time of a non-Markovian process, the step-reinforced random walk (SRRW) on a finite group. This process differs from a classical random walk in that at each integer time, with probability $\alpha$ the next step is chosen uniformly from the previous steps of the walk. We prove that the distribution of the SRRW converges to the uniform distribution exponentially fast if the walk is irreducible and aperiodic. When the step distribution is either symmetric, a class function, or has an atom at the identity, we relate the mixing time of the SRRW to the spectral gap and the mixing time of the underlying walk. For the reinforced (lazy) simple random walk on $L$-cycles, we show that the mixing time undergoes a phase transition at $\alpha=1/2$ and the reinforcement reduces the mixing time to order $L^{1/\alpha}$ for $\alpha>1/2$. On the $d$-dimensional hypercube, the reinforcement slows down mixing, and the SRRW exhibits cutoff as $d\to\infty$ at time $d\log(d)/[F(\alpha)(1-\alpha)]$, with an $O(d)$ window for each fixed $\alpha\in(0,1)$, where $F(\cdot)$ is a hypergeometric function.
  
  \end{abstract}

\maketitle

\section{Introduction}

\subsection{The model and mixing times}

\label{mainressec}

The step-reinforced random walk is a 
process, whose step sequence is generated by an algorithm introduced by Simon \cite{MR0073085} in 1955: At each time step, the walk either replicates a uniformly random historical step or takes a fresh step independent of the past. Step-reinforced random walks in Euclidean space have been studied extensively, with the {\em elephant random walk} being a prominent example. In Euclidean spaces,
 various scaling limits and recurrence-transience criteria have been established, see Section \ref{srrwonRd}. In this paper, we study the walk on finite groups, where it exhibits very different behavior.

\begin{definition}[SRRW on a discrete group]
  \label{defSRRWG}
  Let $(G,\cdot)$ be a discrete group with $|G|\geq 2$ and let $\mu$ be a probability measure on $G$. Let $(\xi_n)_{n\geq 2}$ be i.i.d.\ Bernoulli random variables with success parameter $\alpha\in [0,1]$, and let $(u_n)_{n\geq 2}$ be independent random variables where each $u_n$ is uniformly distributed on $\{1,2,\ldots,n-1\}$. Assume that these two sequences and all fresh samples from $\mu$ below are independent. Define a walk $(S_n)_{n\in \NN}$ and its step sequence $(X_n)_{n\geq 1}$ recursively as follows:
\begin{enumerate}[(i)]
  \item Set $S_0:=x \in G$ and at time $n=1$, sample $X_1$ from $\mu$, set $S_1:=x\cdot X_1$;
  \item For $n> 1$, given $X_1,X_2,\dots,X_{n-1}$: 
  \begin{itemize}
    \item If $\xi_{n}=1$, set $X_{n}:=X_{u_{n}}$;
    \item If $\xi_{n}=0$, sample $X_{n}$ independently from $\mu$.
  \end{itemize}
 Update $S_{n}:=S_{n-1}\cdot X_{n}$.
\end{enumerate}
The process $S=(S_n)_{n\in \NN}$ is called a step-reinforced random walk (SRRW) on $G$ started at $x$, with step distribution $\mu$ and reinforcement parameter $\alpha$. 
\end{definition}

  When $\alpha=0$ (no reinforcement), the walk $S$ reduces to a classical random walk on $G$, with i.i.d.\ steps distributed as $\mu$. When $\alpha=1$, we have $S_n=S_0\cdot (X_1)^n$ for all $n\geq 1$. Also, if $S$ is an SRRW starting from the group identity $e_G$, then for any $x\in G$, the process $x\cdot S$ is an SRRW starting from $x$. We will henceforth assume that an SRRW always starts from $e_G$ and has reinforcement parameter less than 1.  
The main assumption of this paper is the following: 

\begin{assumption}
\label{mainassum}
   Suppose $S=(S_n)_{n\in \NN}$ is an SRRW on a finite group $G$ with parameter $\alpha \in [0,1)$ and step distribution $\mu$ such that the transition matrix $P_{\mu}$ defined below is irreducible and aperiodic (in this case, we shall also say that the walk $S$ is irreducible and aperiodic):
   \begin{equation}
  \label{Pdef}
P_{\mu}(x,y):=\mu(x^{-1}\cdot y), \quad \text{for } x,y\in G.
\end{equation}
\end{assumption}

For irreducible and aperiodic Markov chains on finite groups, or more generally, on finite graphs, a central topic is the convergence rate of the chain's distribution to stationarity, or equivalently, the mixing time. We refer the reader to \cite{MR3726904} for a comprehensive introduction.  Although the SRRW is in general non-Markovian, it is surprising to see that, because of the group structure, an irreducible and aperiodic SRRW on a finite group will quickly ``forget'' its past in the sense that its distribution converges to the uniform distribution exponentially fast.
\begin{proposition}
     \label{conexpthm}
    Under Assumption \ref{mainassum}, there exist two positive constants $C=C(G,\mu)$ and $\rho=\rho(G,\mu,\alpha)\in (0,1)$ such that for any $n\geq 1$,
     \begin{equation}
      \label{mainineTVexp}
      \|\PP(S_{n}=\cdot)- U\|_{\mathrm{TV}} \leq C\rho^{(1-\alpha)n},
    \end{equation}
    where $U$ denotes the uniform measure on $G$. 
\end{proposition}
We believe that the constant $\rho$ in (\ref{mainineTVexp}) can be chosen to be independent of $\alpha \in [0,1)$. In Section \ref{resultsupbd}, we shall prove that this holds under mild conditions. Under Assumption \ref{mainassum}, for any $\varepsilon \in (0,1)$, we define the $\varepsilon$-mixing time $t^{(\alpha)}_{\mathrm{mix}}(\varepsilon)$ by
 \begin{equation}
     \label{tmixdef}
      t_{\mathrm{mix}}^{(\alpha)}(\varepsilon):=\inf\{n\geq 1: \|\PP(S_{m}=\cdot)- U\|_{\mathrm{TV}} \leq \varepsilon, \forall m\geq n\},
 \end{equation}
which is finite by Proposition \ref{conexpthm}. This paper aims to estimate the mixing time in various settings. 

\begin{example}
    \label{examSZ2}
 Consider the case $G=(\ZZ_2, +)$ and $\mu(1)= \mu(0)=1/2$. Then $S_1\sim \operatorname{Unif}\{0,1\}$ and $\PP(S_2=0)=(1+\alpha)/2>1/2$ if $\alpha >0$, which shows that $\|\PP(S_{n}=\cdot)- U\|_{\mathrm{TV}}$ is not necessarily monotone in $n$. Thus, the definition of $t^{(\alpha)}_{\mathrm{mix}}(\varepsilon)$ requires that the total variation distance remains below $\varepsilon$ for all $m\geq t^{(\alpha)}_{\mathrm{mix}}(\varepsilon)$.  
\end{example}

\subsection{Phase transition on cycles and cutoff on hypercubes}
\label{secexpconphasetran}

 \begin{figure}[t]
    \centering
    \includegraphics[width=11cm]{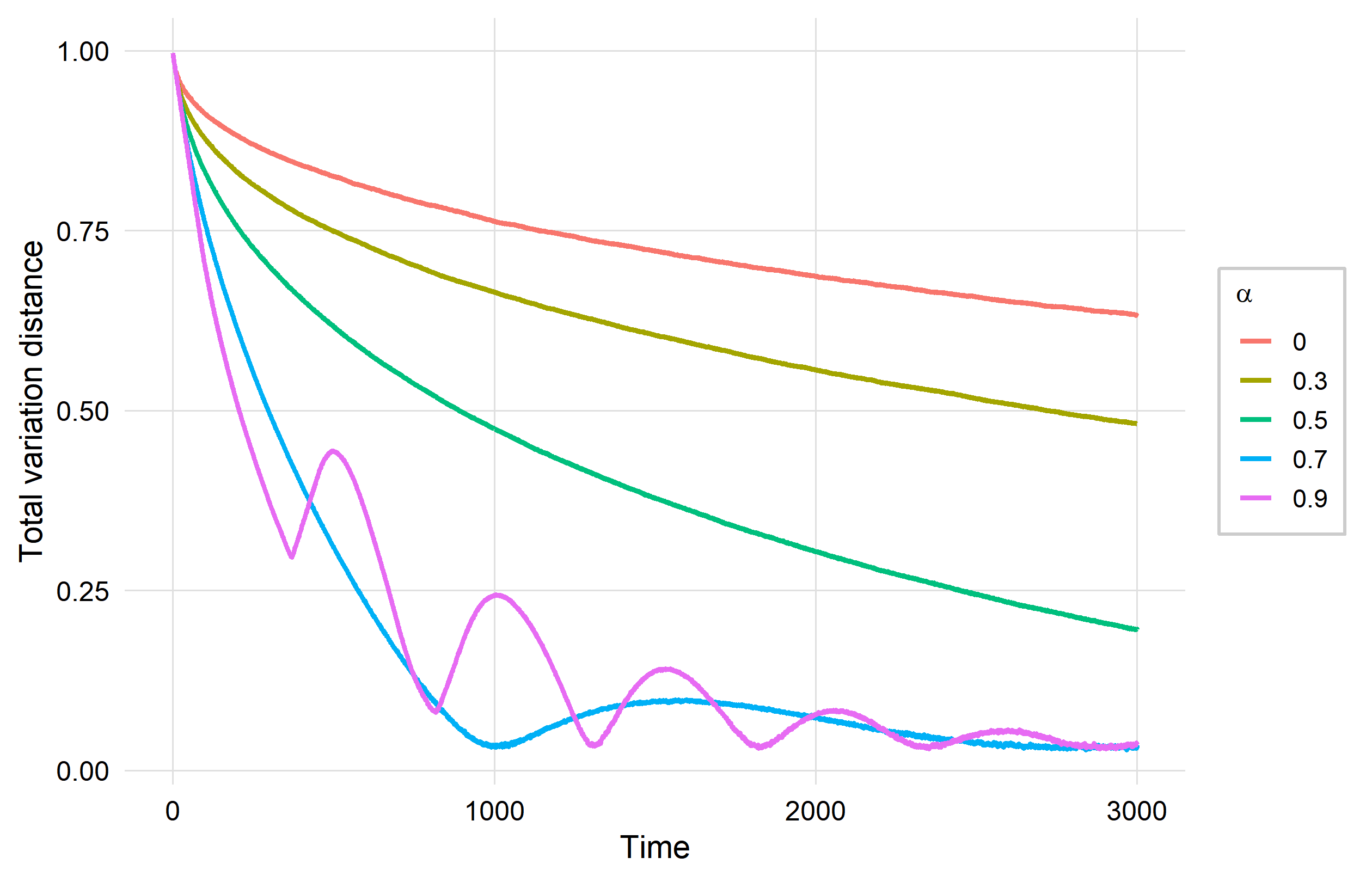}
 \caption{Decay of the total variation distance to stationarity for step-reinforced simple random walks on the cycle $\ZZ_{333}$. The curves are Monte Carlo estimates based on $50000$ independent simulations.}
\label{fig:cycle}
\end{figure} 

The following Theorem \ref{phasetrancycle} shows that in some groups, \emph{step-reinforcement can speed up the mixing}. More specifically, for the reinforced simple random walk on an odd cycle of length $L$, there exists a phase transition:
$$
t^{(\alpha)}_{\mathrm{mix}}(\varepsilon)=\Theta(L^2), \text{ for }\alpha<\frac{1}{2}; \quad \, t^{(\alpha)}_{\mathrm{mix}}(\varepsilon)=\Theta(\frac{L^2}{\log L}), \text{ for }\alpha=\frac{1}{2}; \quad \, t^{(\alpha)}_{\mathrm{mix}}(\varepsilon)=\Theta(L^{\frac{1}{\alpha}}), \text{ for }\alpha>\frac{1}{2}.
$$
Figure \ref{fig:cycle} illustrates this behavior and also shows pronounced non-monotonicity when $\alpha$ is close to $1$. A possible explanation comes from viewing the walk as the projection onto $\mathbb{Z}_L$ of a one-dimensional SRRW $\widehat{S}$. For $\alpha>1/2$, this walk satisfies $\widehat{S}_n/n^\alpha\to W$ almost surely, where $W$ has a continuous symmetric distribution (see (\ref{superalphaW})). The clockwise and counter-clockwise parts of the distribution wrap around the cycle, and their overlap changes with time. This provides a heuristic explanation for the fluctuations observed when $\alpha=0.9$.

\begin{theorem}
    \label{phasetrancycle}
Let $G=(\ZZ_L, +)$ where $L\geq 3$ is an odd number and assume that $\mu(1)=\mu(-1)=1/2$. Let $S$ be an SRRW on $G$ with reinforcement parameter $\alpha \in [0,1)$ and step distribution $\mu$. Fix $\varepsilon\in (0,1)$. Then:
\begin{enumerate}[\bf (i)] 
    \item If $\alpha \in [0,1/2)$, then there exist two positive constants $c_1=c_1(\alpha,\varepsilon)$ and $C_1=C_1(\alpha,\varepsilon)$ such that for all large $L$,
    $$c_1  L^2 \leq t^{(\alpha)}_{\mathrm{mix}}(\varepsilon) \leq C_1 L^2.$$ 
\item If $\alpha = 1/2$, then there exist two positive constants $c_2=c_2(\varepsilon)$ and $C_2=C_2(\varepsilon)$ such that for all large $L$,
$$ \frac{c_2  L^2}{\log L} \le t^{(\alpha)}_{\mathrm{mix}}(\varepsilon)\leq  \frac{C_2  L^2}{\log L}.$$ 
\item If $\alpha \in (1/2,1)$, then there exist two positive constants $c_3=c_3(\alpha,\varepsilon)$ and $C_3=C_3(\alpha,\varepsilon)$ such that for all large $L$,
 $$c_3L^{\frac{1}{\alpha}} \leq  t^{(\alpha)}_{\mathrm{mix}}(\varepsilon) \leq C_3 L^{\frac{1}{\alpha}}.$$
\end{enumerate}
For every fixed $\alpha\in[0,1)$, this family of walks does not exhibit cutoff as $L\to\infty$ through odd integers.
\end{theorem}
\begin{remark}
(i). If the non-reinforced chain $(\alpha=0)$ is the lazy simple random walk, then one can prove similar results for all integers $L\geq 3$ (not necessarily odd) by applying similar arguments. \\
(ii). We shall upper bound the $\ell^2$ distance (see (\ref{defchidis}) for the definition) between the distribution of $S_n$ and the uniform distribution, which implies the desired upper bounds for the total variation distance.  \\
(iii). Theorem \ref{phasetrancycle} shows that for fixed $\varepsilon$ and $\alpha_1<\alpha_2$ with $\alpha_2\geq 1/2$, one has $t^{(\alpha_1)}_{\mathrm{mix}}(\varepsilon) \gg t^{(\alpha_2)}_{\mathrm{mix}}(\varepsilon)$ for large $L$. However, for fixed $L$, the inequality $t^{(\alpha_1)}_{\mathrm{mix}}(\varepsilon) \geq t^{(\alpha_2)}_{\mathrm{mix}}(\varepsilon)$ does not always hold, as shown in Figure \ref{fig:cycle}. 
\end{remark}

 \begin{figure}[t]
    \centering
    \includegraphics[width=11cm]{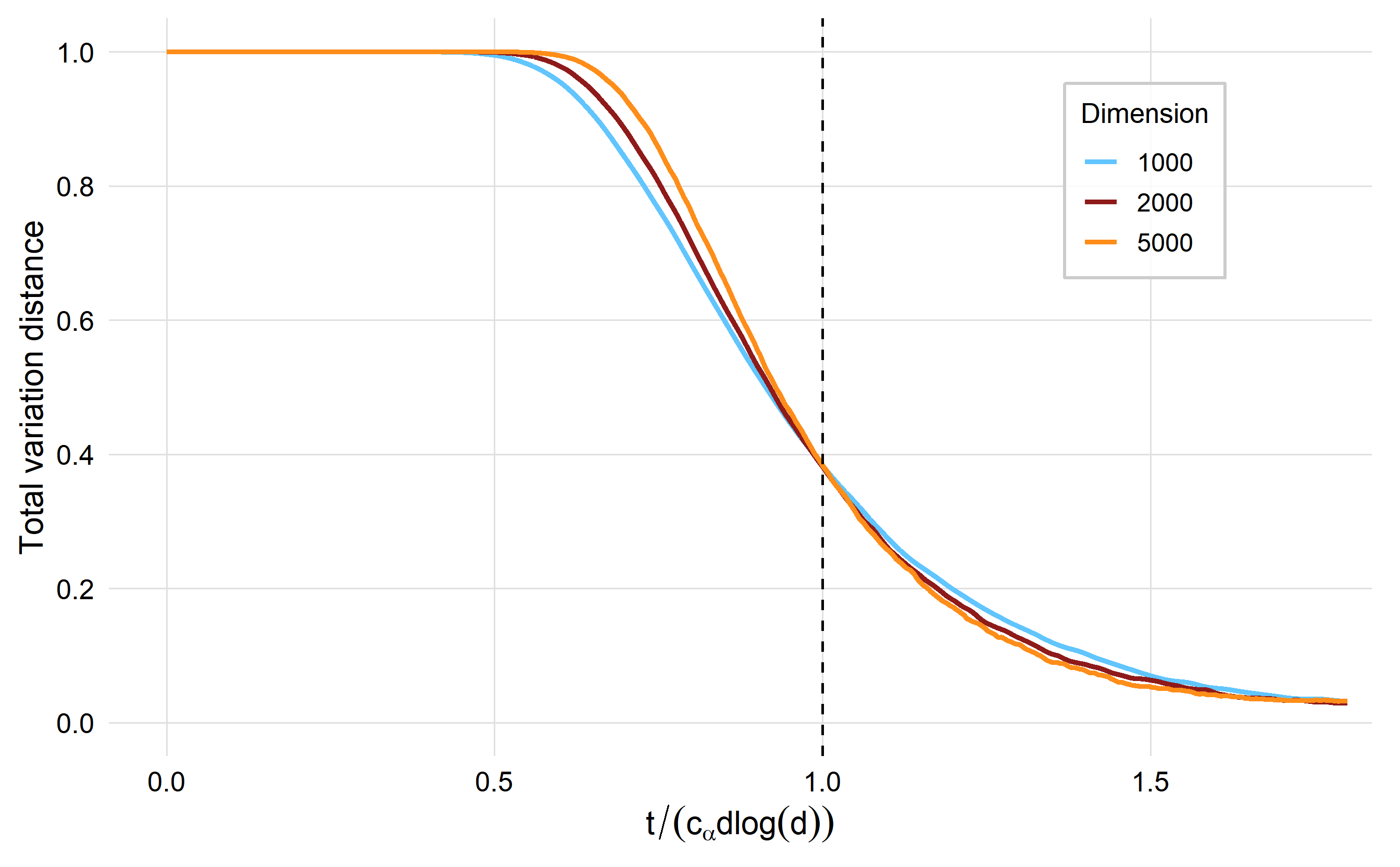}
 \caption{Decay of the total variation distance to stationarity for step-reinforced lazy simple random walks on the hypercube $\mathbb{Z}_2^d$ with $\alpha=0.5$ and $c_{\alpha}=1/[(1-\alpha) \ _2F_1(1,\alpha^{-1};\alpha^{-1}+1; 1/2)] \approx 1.29$. The curves are Monte Carlo estimates based on $30000$ independent simulations and are lightly smoothed using a Gaussian kernel. The estimates use the empirical distribution of the number of $1$'s, since coordinate-permutation invariance implies that the total variation distance of the walk from stationarity equals that of this count from $\operatorname{Bin}(d,1/2)$.}
\label{fig:hypercube}
\end{figure} 

The effect of reinforcement on mixing depends on the group, even in the abelian case. The following Theorem \ref{slowhycu} shows that  \emph{reinforcement slows down the mixing} for lazy random walk on the hypercube $\ZZ_2^d=\{0,1\}^d$, where $d$ is a positive integer. More precisely, we will prove that \emph{the reinforced lazy simple random walk on the hypercube has a cutoff}, that is,   
$$\lim_{d \to \infty} \frac{t^{(\alpha)}_{\mathrm{mix}}(\varepsilon)}{t^{(\alpha)}_{\mathrm{mix}}(1-\varepsilon)} = 1, \quad \forall \varepsilon \in (0,1),$$
see Figure \ref{fig:hypercube} for an illustration. In contrast, the family of walks on cycles in Theorem \ref{phasetrancycle} does not have a cutoff.

In $G=(\ZZ_2^d, +)$, the group identity $e_G$ is the zero vector in $\{0,1\}^d$. For $k=1,2,\dots,d$, we denote by $e_k\in \{0,1\}^d$ the vector with $1$ in the $k$-th position and zeros elsewhere.

\begin{theorem}
      \label{slowhycu}
    Let $S$ be an SRRW on $(\ZZ_2^d, +)$ with reinforcement parameter $\alpha \in (0,1)$ and step distribution $\mu$ given by
     $$\mu(e_G)=\frac{1}{2}, \quad \text{and }\quad \mu(e_k)=\frac{1}{2d}, \ k=1,2,\dots,d.$$
  Then for any fixed $\varepsilon \in (0,1)$, as $d\to\infty$,
$$ t^{(\alpha)}_{\mathrm{mix}}(\varepsilon)
= \frac{d\log d}{(1-\alpha) \ _2F_1(1,\alpha^{-1};\alpha^{-1}+1; \frac{1}{2})}
+O_{\alpha,\varepsilon}(d),$$
where
$$
_2F_1(a,b;c;z):=\sum_{m=0}^{\infty} \frac{(a)_m (b)_m}{(c)_m} \frac{z^m}{m!}
$$
is the hypergeometric function, and $(a)_m$ is the rising Pochhammer symbol given by $(a)_0=1$ and $(a)_m=a(a+1)\cdots (a+m-1)$ for $m\geq 1$.
\end{theorem}
\begin{remark}
It is known that $t^{(0)}_{\mathrm{mix}}(\varepsilon)\sim (d\log d)/2$, where $\sim$ means that the ratio of the two sides tends to 1 as $d \to \infty$. It will be shown in the proof of Corollary \ref{corNJn} that ${}_2 F_1(1,\alpha^{-1};\alpha^{-1}+1; 1/2)$ is a decreasing function of $\alpha$ satisfying
$$
\frac{2}{1+\alpha}< {}_2 F_1(1,\alpha^{-1};\alpha^{-1}+1; \frac{1}{2}) < 2, \quad \forall \alpha \in (0,1).
$$
Consequently, Theorem \ref{slowhycu} shows that if $\alpha>0$, then $t^{(\alpha)}_{\mathrm{mix}}(\varepsilon)>t^{(0)}_{\mathrm{mix}}(\varepsilon)$ for all large $d$. 
\end{remark}

We believe our methods can be extended to general abelian groups. However, studying the cutoff and speed-up phenomena for the SRRW on non‑abelian groups may require new ideas and techniques. In the next section, we provide some quantitative upper bounds for mixing times on general finite groups under mild conditions on the step distribution.

\subsection{Quantitative upper bounds for mixing times}
\label{resultsupbd} 
Besides the total variation distance, one may also use the following $\ell^{\infty}$ distance to study the convergence:
   $$
 d_{\infty}(n):=\max_{x \in G}|\PP(S_{n}=x) \cdot  |G| -1 |.
   $$
    The corresponding mixing time is called the $\varepsilon$-uniform mixing time:
   $$
t_{\infty}^{(\alpha)}(\varepsilon):=\inf\{n\geq 1:  d_{\infty}(m) \leq \varepsilon, \forall m\geq n\}.
   $$
   Note that $\|\PP(S_{n}=\cdot)- U\|_{\mathrm{TV}} \leq d_{\infty}(n)/2$. In particular, Theorem \ref{thmconj} below provides a sufficient condition for choosing $\rho$ in (\ref{mainineTVexp}) to be independent of $\alpha$. We let 
$$
\Gamma:=\{x\in G: \mu(x)>0\} 
$$
be the support of $\mu$, and let $\Gamma^{-1}:=\{x^{-1}: x\in \Gamma\}$. For two subsets $A,B$ of $G$, we write $A\cdot B:=\{a\cdot b: a\in A, b\in B\}$. For a non-empty subset $A$ of $G$, we denote by $\langle A \rangle$ the subgroup generated by $A$.

\begin{theorem}
    \label{thmconj}
 Under Assumption \ref{mainassum}, if $ \langle \Gamma \cdot \Gamma^{-1}  \rangle=G$, then there exist two positive constants $C=C(G,\mu)$ and $\rho=\rho(G,\mu)\in (0,1)$ such that for any $n\geq 1$, 
 $$
d_{\infty}(n) \leq C\rho^{(1-\alpha)n}.
 $$
 In particular, this holds in the following cases:
\begin{enumerate}[\bf (i)]
\item $\Gamma$ is symmetric, i.e., $\Gamma=\Gamma^{-1}$;
\item   $\Gamma$ is  a union of conjugacy classes of $G$, which contains case when $G$ is abelian.
\item  $e_G \in \Gamma$.
\end{enumerate}
\end{theorem}
\begin{remark}
 In Lemma \ref{sizeincconmu}, we will show that under Assumption \ref{mainassum} the equality $\langle \Gamma \cdot \Gamma^{-1}  \rangle=G$ holds if and only if   $ \langle \Gamma \cdot \Gamma^{-1}  \rangle=\langle \Gamma^{-1} \cdot \Gamma  \rangle$. These equalities do not always hold: if $G$ is the symmetric group on $\{1,2,3\}$ and $\Gamma=\{(12),(132)\}$. Then, $ \Gamma^{-1}=\{(12),(123)\}$ and
    $$
  \langle \Gamma \cdot \Gamma^{-1}  \rangle=\{e_G, (13)\} \neq \{e_G, (23)\}= \langle \Gamma^{-1} \cdot \Gamma  \rangle.
    $$
\end{remark}

Proposition \ref{propabelcompare} below shows that $t^{(\alpha)}_{\mathrm{mix}}(\varepsilon)$ can be upper bounded by studying the underlying non-reinforced chain if $\mu$ is a class function,  i.e., it is constant on conjugacy classes: 
$$ \mu(g^{-1}\cdot x\cdot g)=\mu(x), \quad  \forall x,g \in G, $$
or equivalently, $\mu( x\cdot g)=\mu(g\cdot x)$ for all $x,g \in G$. Note that every probability measure on an abelian group is a class function. We write $t^{(\alpha, G, \mu)}_{\mathrm{mix}}(\varepsilon)$ to indicate the dependence of $t^{(\alpha)}_{\mathrm{mix}}(\varepsilon)$ on $(G,\mu)$.

\begin{proposition}
    \label{propabelcompare}
 (i). Under Assumption \ref{mainassum}, if $\mu$ is a class function, then for any $\varepsilon \in (0,1)$,
       \begin{equation}
      \label{ineAbelmixcom}
 t^{(\alpha)}_{\mathrm{mix}}(\varepsilon) \leq  \frac{1}{1-\alpha} \max\left\{8t^{(0)}_{\mathrm{mix}}\left(\frac{\varepsilon}{2}\right),\ 18\log \left(\frac{4}{\varepsilon}\right)  \right\}+1.
    \end{equation}
 (ii). Assume that for each $n\geq 1$, $\mu_n$ is a probability measure on a finite group $G_n$ such that $P_{\mu_n}$ is irreducible and aperiodic and $\mu_n$ is a class function. If for any $\varepsilon \in (0,1)$, we have $t^{(0, G_n,\mu_n)}_{\mathrm{mix}}\left(\varepsilon/2\right) \to \infty$ as $n\to \infty$, then for any $\alpha \in [0,1)$,
 $$
\limsup_{n\to \infty}\frac{t^{(\alpha,G_n,\mu_n)}_{\mathrm{mix}}(\varepsilon)}{t^{(0, G_n,\mu_n)}_{\mathrm{mix}}(\varepsilon/2)} \leq\frac{1+\alpha}{1-\alpha}.
 $$
\end{proposition}
\begin{remark}
The second term in the maximum in (\ref{ineAbelmixcom}) cannot be removed. In the companion paper \cite{peres2026transition}, we show that for Example \ref{examSZ2} (the group of 2 elements), for any $\alpha\in(0,1)$ and any  $n\geq 1$,
$$
\|\PP(S_{2n}=\cdot)-U\|_{\mathrm{TV}}\geq\frac12\alpha^n e^{-(1-\alpha)n}.
$$
Consequently, for any $\varepsilon\in(0,1/2)$,
$$
t^{(\alpha)}_{\mathrm{mix}}(\varepsilon)\geq\frac{2}{1-\alpha-\log\alpha}\log\left(\frac{1}{2\varepsilon}\right),
$$
Note that $t^{(0)}_{\mathrm{mix}}(\varepsilon)=1$ since $S_1\sim\operatorname{Unif}\{0,1\}$. Since $1-\alpha-\log\alpha\sim2(1-\alpha)$ as $\alpha\to1$, this also shows that the factor $(1-\alpha)$ in Theorem \ref{thmconj} cannot be improved in general (when $\alpha$ is close to $1$).
\end{remark}

We say $\mu$ is symmetric if $\mu(g)=\mu(g^{-1})$ for any $g\in G$, which, in our setting, is equivalent to the underlying chain being reversible. The spectrum of the matrix $P_{\mu}$ is denoted by $\operatorname{spec}(P_{\mu})$, and under Assumption \ref{mainassum}, it is well known that 
$$
\lambda_*:=\max\{|\lambda|: \lambda \in \operatorname{spec}(P_{\mu}), \lambda \neq 1\}<1.
$$
The difference $\gamma_*:= 1- \lambda_*>0$ is called the absolute spectral gap of $P_{\mu}$. 

\begin{proposition}
\label{propclassf}
   Under Assumption \ref{mainassum}, if $\mu$ is symmetric, then for any $\varepsilon \in (0,1)$,
       \begin{equation}
      \label{ineAbelmix}
 t^{(\alpha)}_{\mathrm{mix}}(\varepsilon) \leq  \frac{C}{1-\alpha} \log \left(\frac{|G|}{\varepsilon}\right) \frac{1}{\gamma_*},
    \end{equation}
   where $\gamma_*$ is the absolute spectral gap of $P_{\mu}$, and $C$ is a positive absolute constant that does not depend on $G,\mu, \alpha$ and $\varepsilon$.
\end{proposition}

\begin{remark}
(i). In the Markovian case, see e.g. \cite[Theorem 12.4]{MR3726904}, if $\mu$ is symmetric, then 
$$
t^{(0)}_{\mathrm{mix}}(\varepsilon) \leq \log \left(\frac{|G|}{\varepsilon}\right) \frac{1}{\gamma_*}.
$$
Proposition \ref{propclassf} extends this result to the reinforced case $(\alpha>0)$ up to a factor $C/(1-\alpha)$. \\
(ii).  Neither of the two conditions in Propositions \ref{propclassf} and \ref{propabelcompare} (i) implies the other: 
\begin{itemize}
    \item Let $G$ be the symmetric group on $\{1,2,3\}$ with conjugate classes $\{e_G\},\{(12),(13),(23)\}$ and $\{(123),(132)\}$. Assume that $\mu$ is a probability measure on $G$ such that 
    $$\mu(e_G)>\mu((12))>\mu((13))>\mu((23))>0, \quad \mu((123))=\mu((132))=0.$$ 
    Then $P_{\mu}$ is irreducible and aperiodic but $\mu$ is not a class function.
    \item Let $G$ be a group of odd order. Then only $e_G$ is conjugate to its inverse. Let $\mu$ be a positive class function on $G$ such that it takes different values on different conjugate classes, then it is not symmetric.
\end{itemize}
\end{remark}

If $\mu(e_G)$ is positive, Proposition \ref{proplazy} below shows that the $\varepsilon$-uniform mixing time can be upper bounded using the isoperimetric profile. Let us introduce some preliminary notation. For two subsets $A,B$ of a countable group $G$, we write 
\begin{equation}
    \label{PmuABdef}
    P_{\mu}(A,B):=\sum_{x\in A,y\in B} P_{\mu}(x,y).
\end{equation}
Following \cite{MR3726904}, for a non-empty subset $A \subset G$, we call $\Phi(A):=P_{\mu}(A,A^c)/|A|$ the bottleneck ratio of $A$. When $G$ is finite, we define the isoperimetric profile $\Phi(r)$ for $r\geq 1/|G|$ by 
\begin{equation}
    \label{defphir}
    \Phi(r):=\inf \{\Phi(A): U(A) \leq r\},\quad r \in\left[\frac{1}{|G|}, \frac{1}{2} \right]; \quad \Phi(r):=\Phi\left(\frac{1}{2}\right), \quad r> \frac{1}{2},
\end{equation}
where $U(A):=|A|/|G|$. We note that, in the literature, the constant $\Phi(1/2)$ is called the bottleneck ratio of the Markov chain with transition matrix $P_{\mu}$, or conductance, or Cheeger constant, or isoperimetric constant. 
\begin{proposition}
    \label{proplazy}
  Let $S$ be an SRRW on a finite group $G$ with parameter $\alpha \in [0,1)$ and step distribution $\mu$ such that $\mu(e_G)\geq \mu_0$ for some $\mu_0 \in (0,1/2]$. Then, for any $\varepsilon \in (0,1)$, 
    $$
t_{\infty}^{(\alpha)}(\varepsilon) \leq \frac{C(\mu_0)}{1-\alpha} \int_{4/|G|}^{8/ \varepsilon} \frac{1 }{u \Phi^2(u)} d u.
  $$
  where $C(\mu_0)$ is a positive constant that depends only on $\mu_0$.
\end{proposition}

\begin{example}
  Consider the lamplighter group $G=\ZZ_2^{\ZZ_L} \times \ZZ_L$ $(L\geq 3)$ with group operation
    $$
 (f,j) \cdot (h,k):=(\phi,j+k), \text{where } \phi(i):=f(i)+h(i-j) \mod 2, \forall i\in \ZZ_L.
    $$
   Define $h_0: \ZZ_L \mapsto \ZZ_2$ and $h_1: \ZZ_L \mapsto \ZZ_2$ by
   $$h_0(0)=0, h_1(0)=1, \text{ and } h_0(i)=h_1(i)=0, \forall \ i\neq 0.$$ 
   Note that $e_G=(h_0,0)$. Define a probability measure $\mu$ by
    $$
\mu(e_G)=\frac{1}{2}, \mu((h_1,0))=\frac{1}{4},  \mu((h_0,1))=\mu((h_0,-1))=\frac{1}{8}.
    $$
    Let $S$ be an SRRW on $G$ with step distribution $\mu$ and reinforcement parameter $\alpha \in [0,1)$. When $\alpha=0$, the chain admits the following interpretation: Each vertex in $\ZZ_L$ (the cycle of length $L$) is equipped with a lamp that can be either on (state $1$) or off (state $0$). A lamplighter is positioned at a vertex. At each time step, with probability $1/2$, the lamplighter does nothing; with probability $1/4$, it switches the lamp at its current location; with probability $1/4$, it moves at random to one of the two adjacent lamps.
    
It is known that there exist constants $C_1>0$ and $C_2>1$ such that for all $L$,
$$
\Phi(r) \geq \frac{C_1}{\log (C_2 r|G|)}, \quad \text{for } r \in\left[\frac{1}{|G|}, \frac{1}{2} \right]
$$
Since $|G|=L2^L$, Proposition \ref{proplazy} then implies that there exists a positive constant $C_3$ such that for all $\alpha$ and $L$,
$$
t_{\infty}^{(\alpha)}(\frac{1}{4}) \leq \frac{C_3 L^3}{1-\alpha}.
$$
For comparison, the non-reinforced walk satisfies $t_{\infty}^{(0)}(1/4)\geq C_4L^3$ for some absolute constant $C_4>0$, see \cite[Example 4]{MR2198701}.
\end{example}

\subsection{Previous results on Euclidean spaces}
\label{srrwonRd}

Definition \ref{defSRRWG} also applies to measurable groups, in particular to $(\RR^d,+)$ with a Borel step distribution $\mu$. Most previous work concerns this Euclidean setting. For other discrete groups, we refer to Mukherjee \cite{mukherjee2025elephant} on infinite Cayley trees and Mukherjee and Talukdar \cite{mukherjee2026dihedral} on the infinite dihedral group. We recall the scaling limits that will be used in the proof of Theorem \ref{phasetrancycle}.

When $\mu$ is the uniform distribution on the set $\left\{ -1,+1\right\}$, the SRRW is the so-called elephant random walk (ERW) introduced by Schütz and Trimper \cite{schutz2004elephants}. For $\alpha \leq 1/2$, one has the following asymptotic normality:
\begin{equation}
\label{CLTERW}
  \frac{S_n}{\sqrt{n}} \overset{d}{\longrightarrow} \mathcal{N}\left(0,\frac{1}{1-2\alpha}\right),\ \text{if }\alpha <\frac{1}{2}; \quad \frac{S_n}{\sqrt{n \log n}} \overset{d}{\longrightarrow} \mathcal{N}(0,1), \ \text{if }\alpha =\frac{1}{2};
\end{equation}
and for $\alpha>1/2$, one has the following almost-sure convergence:
\begin{equation}
    \label{superalphaW}
      \lim _{n \rightarrow \infty} \frac{S_n}{n^{\alpha}}=W, \quad \text{almost surely,}
\end{equation}
where $W$ is a non-degenerate random variable, see \cite{baur2016elephant, MR3741953, MR3652225}. The distribution of $W$ has been studied in depth by Gu\'erin et al.\ \cite{MR5079601, MR4926011}.

Bercu and Laulin \cite{MR3962977} studied the multidimensional ERW with step distribution uniform on $\{\pm e_1,\ldots,\pm e_d\}$, while Businger \cite{MR3827299} studied the shark random swim with isotropic stable steps. For centered one-dimensional step distributions with finite second moment, Bertoin \cite{MR4237267} and Bertenghi and Rosales-Ortiz \cite{MR4490996} established invariance principles in the diffusive and critical regimes, respectively. Hu \cite{hu2025berry} obtained Berry--Esseen bounds. Bertenghi and Rosales-Ortiz \cite{MR4490996} also proved a law of large numbers under a second moment assumption, which Hu and Zhang \cite{MR4794980} relaxed to a first moment assumption. For centered steps with finite second moment in the superdiffusive regime, extensions of (\ref{superalphaW}) were obtained by Bertenghi \cite{bertenghi2021asymptotic} and Bertoin \cite{MR4237267}. Under a finite $2+\delta$ moment assumption, Qin \cite{MR5009036} proved a recurrence-transience phase transition at $\alpha=1/2$ for genuinely $d$-dimensional centered step distributions in dimensions $d=1,2$, and transience for all $\alpha\in[0,1)$ in dimensions $d\geq3$. Transition probability estimates on infinite groups are studied in the companion paper \cite{peres2026transition}.

\subsection{Preliminaries and notation}
\label{prelnotasec}

For a positive integer $n$, we write $[n]:=\{1,2,\dots,n\}$. For nonnegative functions $f(n), g(n)$ of $n \in \mathbb{Z}_{+}$, we write $f(n)=\Theta(g(n))$ if there exist two positive constants $C_1$ and $C_2$ such that $C_1 f(n) \leq g(n) \leq C_2f(n)$ for all large $n$.

We let $C(a_1, a_2,..., a_k)$ and $c(a_1, a_2,..., a_k)$ denote a positive constant depending only on variables $a_1, a_2,..., a_k$. For example, $C(G,\mu)$ in Proposition \ref{conexpthm} denotes a constant that depends on the group $G$ and step distribution $\mu$ but does not depend on the reinforcement parameter $\alpha$. The actual values of these constants may vary from line to line.

For any two probability measures $\nu_1,\nu_2$ on a finite group $G$, the total variation distance between $\nu_1$ and $\nu_2$ is defined as
$$
\|\nu_1-\nu_2\|_{\mathrm{TV}}:=\sup_{A \subset G}|\nu_1(A)-\nu_2(A)|= \frac{1}{2} \sum_{g \in G}|\nu_1(g)-\nu_2(g)|.
$$
If $\nu_2$ is positive, we let $\chi(\nu_1, \nu_2)$ be the $\ell^2$-distance between $\nu_1$ and $\nu_2$ with respect to $\nu_2$:
\begin{equation}
\label{defchidis}
    \chi^2(\nu_1, \nu_2):=\sum_{g \in G} \nu_2(g)\left(\frac{\nu_1(g)}{\nu_2(g)}-1\right)^2=\left(\sum_{g \in G} \frac{\nu_1(g)^2}{\nu_2(g)}\right)-1.
\end{equation}
Note that $\chi(\nu_1, \nu_2)\geq 2\|\nu_1-\nu_2\|_{\mathrm{TV}}$.

\subsection{Organization of the paper}

The remainder of this paper is organized as follows. 

In Section \ref{secsrrwrrt}, using a connection to the percolated random recursive tree, we express the SRRW as a mixture of time-inhomogeneous Markov chains. We study the sizes of the clusters in the percolated random recursive tree in Section \ref{secclustersize}, which is of independent interest. Using these estimates, we prove the main results in Section \ref{proofsec}. More precisely:
\begin{itemize}
    \item We prove Proposition \ref{propabelcompare} in Section \ref{proofpropabelcompare}, using the lower bound for the number of isolated vertices after percolation.
    \item  By estimating the number of consecutive isolated vertices and using a contraction argument, we prove Proposition \ref{conexpthm} in Section \ref{contractkersec}. 
 \item Proposition \ref{propclassf} is proved in Section \ref{specsec} via spectral techniques. 
 \item Proposition \ref{proplazy} and Theorem \ref{thmconj} are proved in Section \ref{secevosets} using the evolving set process.
    \item In Section \ref{secabel}, we prove Theorem \ref{phasetrancycle} by Fourier analysis. To estimate the Fourier coefficients, we shall use a result obtained in Section \ref{secclustersize} that there are ``sufficiently many" clusters of various scales in the percolated random recursive tree.
    \item In Section \ref{secerwpoly}, we prove Theorem \ref{slowhycu} using the coupling method.
\end{itemize}

\section{SRRW as a mixture of time-inhomogeneous Markov chains}
\label{secsrrwrrt}

K\"ursten \cite{MR3652690} and Businger \cite{MR3827299} pointed out that two special cases of SRRWs, i.e., the elephant random walk and the shark random swim, have a connection to Bernoulli bond percolation on random recursive trees. This connection still holds for the general SRRW on $\RR^d$ and has been used in \cite{MR4237267, MR5009036, hu2025berry}, see e.g. \cite[Section 2.4]{MR5009036} for a short introduction. We generalize this connection in the setting of groups and use it to express the SRRW as a mixture of time-inhomogeneous Markov chains.

Let $(G, \cdot)$, $\mu$, $\alpha$ and $(\xi_n)_{n\geq 2}$, $(u_n)_{n\geq 2}$ be as in Definition \ref{defSRRWG}. Let $(g_n)_{n\geq 1}$ be i.i.d. $\mu$-distributed random elements, independent of $(\xi_n)_{n\geq2}$ and $(u_n)_{n\geq2}$. Given $(\xi_n)_{n\geq 2}$ and $(u_n)_{n\geq 2}$, we construct a growing random forest $(\mathscr{F}_n)_{n\geq 1}$ and assign ``spins" $(g_n)_{n\geq 1}$ to its components as follows: At time $n=1$, there is a vertex with label 1. We denote by $\mathscr{F}_1$ the forest with this single vertex. Later, at each time step $n\geq 2$:
\begin{enumerate}[(i)] 
  \item  We add and connect a new vertex labeled $n$ to the node $u_{n}$ in  $\mathscr{F}_{n-1}$.
 \item  If $\xi_n=0$, the edge connecting the new vertex to the existing vertex is deleted; and if $\xi_n=1$, the edge is retained. We then get a forest with $n$ vertices, which we denote by $\mathscr{F}_n$. 
 \item In each connected component of $\mathscr{F}_n$, we designate the vertex with the smallest label as the root. For $j\in [n]$, we denote by $\CalC_{j, n}$ the cluster rooted at $j$ and denote by $\left|\CalC_{j, n}\right|$ its size, with the convention that $\CalC_{j, n}=\emptyset$ if there is no cluster rooted at $j$. To each cluster $\CalC_{j,n}$, we assign a spin $g_j$. 
\end{enumerate}

For each positive integer $k$, we let $L(k):=j$ if the vertex with label $k$ belongs to $\CalC_{j,k}$ (or equivalently, $\CalC_{j,n}$ for any $n\geq k$). Observe that, for any $n\geq j\geq 1$, the component $\CalC_{j, n}\neq \emptyset$ if and only if $\xi_j=0$ (with the convention that $\xi_1\equiv 0$). In particular, the root of $\CalC_{j,n}$ and the spin assigned to $\CalC_{j,n}$ do not change as $n$ increases, though $\CalC_{j,n}$ may grow as $n$ increases. 

The following Proposition \ref{consSRRWRRT} shows that one can obtain an SRRW by multiplying those spins in order, see Fig. \ref{forest7} for an illustration. We note that the group $G$ does not need to be finite.
\begin{proposition}
\label{consSRRWRRT}
    Define a random walk $S=(S_n)_{n\in \NN}$ on $G$ by $S_0:=e_G$ and 
 \begin{equation}
     \label{rrtconstrS}
      S_n := g_{L(1)}\cdot g_{L(2)}\cdots g_{L(n)}, \quad n\geq 1.
 \end{equation}
 Then $S$ is an SRRW with step distribution $\mu$ and parameter $\alpha$. 
\end{proposition}
\begin{proof}
    First note that $S_1=g_1$ by definition, which has distribution $\mu$. It remains to check that for any $n\geq 1$, given $S_0,S_1,\dots,S_n$, the distribution of $S_{n+1}$ satisfies 
    \begin{equation}
        \label{checkSnsrrw}
         \PP(S_{n}^{-1}\cdot S_{n+1}\in B \mid S_0,S_1,\dots,S_n) =  (1-\alpha)\mu(B) + \alpha \widehat{\mu}_n(B), \quad \text{for any measurable set }B,
    \end{equation}
    where $\widehat{\mu}_n$ is the empirical distribution of the steps of $S$ up to time $n$, i.e.,
    $$
 \widehat{\mu}_{n}:= \frac{1}{n} \sum_{i=1}^n \delta_{S_{i-1}^{-1}\cdot S_i}, \quad n\geq 1.
    $$
    By definition, one has $S_{n}^{-1}\cdot S_{n+1}=g_{L(n+1)}$, and thus,
    $$
    \begin{aligned}
     &\quad\ \PP(S_{n}^{-1}\cdot S_{n+1}\in B  \mid  \F_{n}, (g_{j})_{j\in [n]})\\
     &=\EE \left( \mathds{1}_{\{L(n+1)=n+1\}}\mathds{1}_{\{g_{n+1}\in B\}} +\sum_{\ell=1}^n \mathds{1}_{\{L(n+1)=\ell\}}\mathds{1}_{\{g_{\ell}\in B\}} \mid  \F_{n}, (g_{j})_{j\in [n]} \right) \\
     &=(1-\alpha)\mu(B)+\sum_{\ell=1}^n \frac{\alpha |\CalC_{\ell,n}|}{n} \mathds{1}_{\{g_{\ell}\in B\}}=(1-\alpha)\mu(B)+ \alpha \widehat{\mu}_n(B),
    \end{aligned}
    $$
where we used the fact that $\sum_{\ell=1}^n  |\CalC_{\ell,n}| \mathds{1}_{\{g_{\ell}\in B\}}$ counts the total number of steps $S_{i-1}^{-1}\cdot S_i$ $(i=1,2,\dots,n)$ which belong to $B$. Since $S_0,S_1,\dots,S_n$ are measurable with respect to the sigma-algebra generated by $\F_{n}$ and $(g_{j})_{j\in [n]}$, the equality (\ref{checkSnsrrw}) follows from the tower property of conditional expectation.
\end{proof}

     \begin{figure}
    \centering
    \begin{tikzpicture}[scale=1.5]
		\node[circle,draw, minimum size=0.5cm] (A) at  (0,0) {1};
		\node[circle,draw, minimum size=0.5cm] (B) at  (-0.5,1)  {2};
		\node[circle,draw, minimum size=0.5cm] (C) at  (0.9,1)  {3};
		\node[circle,draw, minimum size=0.5cm] (D) at  (-1,2.1)  {4};
		\node[circle,draw, minimum size=0.5cm] (E) at  (1.6,2.2)  {5};
		\node[circle,draw, minimum size=0.5cm] (F) at  (-1.5,3.1)  {6};
		\node[circle,draw, minimum size=0.5cm] (G) at  (-0.5,3.1)  {7}; 
		\draw (A) -- (B);
		\draw[dotted] (A) -- (C);
		\draw[dotted] (B) -- (D);
		\draw[dotted] (C) -- (E);
		\draw (D) -- (F);
		\draw (D) -- (G);
		\node[left] at (-0.9,1.4) {$\xi_4=0$};
		\node[right] at (0.4,0.2) {$\xi_3=0$};
        \node[right] at (1.5,1.5) {$\xi_5=0$};
        \node[left,red] at (-0.2,0.4) {$g_1$};
        \node[left,purple] at (-1.2,2.4) {$g_4$};
        \node[right,blue] at (1.1,0.9) {$g_3$};
        \node[above,cyan] at (1.6,2.4) {$g_5$};
        \begin{scope}[shift={(-0.3,0.5)},rotate=115]
    \draw[red,thick,dashed] (0,0) ellipse (1.1cm and 0.6cm);
\end{scope}
\begin{scope}[shift={(0.95,1)},rotate=65]
    \draw[blue,thick,dashed] (0,0) ellipse (0.6cm and 0.6cm);
\end{scope}
 \begin{scope}[shift={(-1,2.75)}]
    \draw[purple,thick, dashed] (0,0) circle (1.1cm);
\end{scope}
\begin{scope}[shift={(1.6, 2.3)}]
    \draw[cyan,thick, dashed] (0,0) circle (0.6cm);
\end{scope}
		\end{tikzpicture}
    \caption{An illustration of $S_7$ and the forest $\mathscr{F}_7$ where $u_2=u_3=1$, $u_4=2$, $u_5=3$, $u_6=u_7=4$ and $S_7=g_1^2\cdot g_3\cdot g_4\cdot g_5 \cdot g_4^2$.}
    \label{forest7}
\end{figure}
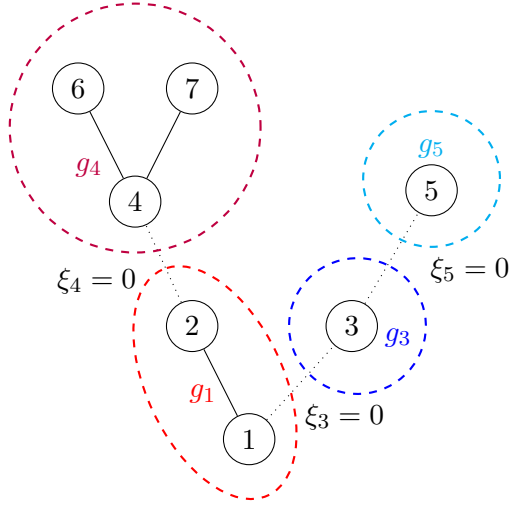

For $n\geq 1$, let $\mathscr{I}_n:=\{1\leq j \leq n: |\CalC_{j,n}|=1\}$ be the set of isolated vertices in $\F_n$. In particular, one has $L(j)=j$ for any $j\in \mathscr{I}_n$. Recall that $g_j$ $(j\in [n])$ is the spin assigned to the cluster $\CalC_{j,n}$.  Then by Proposition \ref{consSRRWRRT}, conditionally on $\sigma(\F_n, (g_j)_{j \in [n]\backslash \mathscr{I}_n })$, the SRRW $(S_j)_{0\leq j \leq n}$ is a time-inhomogeneous Markov chain which, at time step $j$, takes a fresh step sampled from $\mu$ if $j\in \mathscr{I}_n$, and takes a (deterministic) step $g_{L(j)}$ if $j\in [n]\backslash \mathscr{I}_n$. We denote the transition probabilities of the chain by $(P_{k,\ell}(x,y))_{0\leq k \leq \ell \leq n, x,y \in G}$, that is, 
    \begin{equation}
        \label{defPkell}
         P_{k,\ell}(x,y) := \PP(S_{\ell}=y \mid S_k=x, \F_n, (g_j)_{j \in [n]\backslash \mathscr{I}_n}).
    \end{equation}
  For $j\in [n]$, we write 
      \begin{equation}
        \label{defPj}
       P_{j}:=P_{j-1,j}.
    \end{equation}
We note that each $P_j$ is either $P_{\mu}$ or $P^{(g)}$ for some $g\in \Gamma$ (recall that $\Gamma$ is the support of $\mu$) depending on whether $j\in \mathscr{I}_n$ or not, where $P^{(g)}$ is defined by
\begin{equation}
    \label{defPgmatrix}
    P^{(g)}(x,y):=\begin{cases}
    1 & \text{if } y = x\cdot g , \\ 
0 & \text{otherwise,}
\end{cases}
\end{equation}
which is the transition matrix corresponding to a deterministic step $g$. When $G$ is finite, we can write
 \begin{equation}
        \label{defPkellsec}
 P_{k,\ell}=\prod_{j=k+1}^{\ell}P_j:= P_{k+1}P_{k+2}\cdots P_{\ell}, \quad \text{ for } 0\leq k < \ell \leq n,
   \end{equation}
where the right-hand side is the usual matrix multiplication. In particular, 
$$P_{0,n}(e_G,\cdot)=\delta_{e_G}P_1P_2\cdots P_n.$$
Here $\delta_z(\cdot)$ is the vector on $G$ which takes the value $1$ at $z$ and $0$ elsewhere.

\begin{proposition}
\label{propcondSnU}
Let $S$ be as in Proposition \ref{consSRRWRRT} and assume that $G$ is finite.   For $n\geq 1$, one has 
    $$
\|\PP(S_{n} = \cdot \mid \F_n, (g_j)_{j \in [n]\backslash \mathscr{I}_n}) -  U\|_{\mathrm{TV}} = \| \delta_{e_G}\prod_{k=1}^nP_k-U\prod_{k=1}^nP_k  \|_{\mathrm{TV}}, 
    $$
    where $\delta_{e_G}$ and $U$ are viewed as row vectors. In particular, 
   $$
\|\PP(S_{n}=\cdot)- U\|_{\mathrm{TV}} \leq \EE \| \delta_{e_G}\prod_{k=1}^nP_k-U\prod_{k=1}^nP_k  \|_{\mathrm{TV}}.
   $$
\end{proposition}
\begin{proof}
The equality follows from the definition of $P_{0,n}$ and the fact that the uniform measure $U$ is stationary for any $P_{k,\ell}$. Now observe that 
    $$
\begin{aligned}
    \|\PP(S_{n}=\cdot)- U\|_{\mathrm{TV}}&=\frac{1}{2} \sum_{g \in G}\left|\EE \left(\PP(S_{n} = g \mid \F_n, (g_j)_{j \in [n]\backslash \mathscr{I}_n})-\frac{1}{|G|}\right)\right| \\
    &\leq \EE \|\PP(S_{n} = \cdot \mid \F_n, (g_j)_{j \in [n]\backslash \mathscr{I}_n}) -  U\|_{\mathrm{TV}},
\end{aligned}
$$
which proves the second assertion.
\end{proof}

\section{Cluster sizes in percolated random recursive trees}
\label{secclustersize}

Proposition \ref{consSRRWRRT} shows that the SRRW is closely related to the percolated random recursive tree since for fixed $n$, the random forest $\F_n$ can be obtained as follows:
\begin{itemize}
    \item first sample $(u_j)_{2\leq j \leq n}$ to get a random recursive tree of size $n$: We start from a root node with label $1$, and for each $j \in \{2,3,\dots, n\}$, we connect $j$ to $u_j$.
    \item then sample $(\xi_j)_{2\leq j \leq n}$ to perform a Bernoulli bond percolation on this tree, more precisely, each edge $(j,u_j)$ is removed if $\xi_j=0$, and otherwise retained. The resulting graph is  $\F_n$.
\end{itemize}

Proposition \ref{propcondSnU} then enables us to apply some techniques developed for time-inhomogeneous Markov chains once we have some control on $\F_n$. For example, if $G$ is an abelian additive group, then (\ref{rrtconstrS}) becomes  
\begin{equation}
    \label{Sndecomabel}
    S_n=\sum_{j=1}^n \left|\CalC_{j, n}\right| g_j, \quad n\geq 1,
\end{equation}
where $\left|\CalC_{j, n}\right|$ denotes the size of the cluster in the forest $\F_n$ rooted at $j$, and $(g_j)_{j\geq 1}$ are i.i.d. $\mu$-distributed random variables independent of $\F_n$. So it is natural to study the sizes of the clusters in the percolated random recursive tree $\F_n$, which is of independent interest.

For $k\geq 1$, let $N_k(n)$ be the number of clusters of size $k$ in $\F_n$. If $\alpha=0$, then $N_1(n)=n$ and $N_k(n)=0$ for any $k\geq 2$. We also use $I(n)$ for $N_1(n)$($=|\mathscr{I}_n|$), which counts the number of isolated vertices in $\F_n$ and will be used to study general (not necessarily abelian) groups. It is known that if $\alpha \in (0,1)$, then for each $k\geq 1$, the number $N_k(n)$ has asymptotically linear growth: Almost surely, 
  \begin{equation}
      \label{limNkn}
      \lim_{n\to \infty} \frac{N_k(n)}{n} = \lim_{n\to \infty} \frac{\EE N_k(n)}{n}=\theta_k:= \frac{1-\alpha}{\alpha} B(k,1+\frac{1}{\alpha}),
  \end{equation}
where $B(\cdot,\cdot)$ is the Beta function. The limits of $\EE N_k/n$, $k=1,2,\dots$, were derived by Simon \cite{MR0073085} and the almost sure convergence was proved by Bertoin, see \cite[Section 5.3]{MR4624056}. It is worth mentioning that the corresponding results for site percolation on random recursive trees were obtained by Gu and Yuan \cite{Gu2026}. The following proposition gives some concentration inequalities for $N_k(n)$, especially for $k=1$, which will be used frequently later. Note that $\theta_1=(1-\alpha)/(1+\alpha)$ which holds when $\alpha=0$ as well.

 \begin{proposition}
        \label{Inestprop}
       For any $\alpha \in (0,1)$, any $n\geq k\geq 1$ and $\varepsilon>0$, one has
\begin{equation}
  \label{expineNkn}
     \PP\left(  \left|\frac{N_k(n)}{n}  - \theta_k\right| \geq  \varepsilon\right) \leq C_1e^{-C_2 n}, 
  \end{equation}
where $C_1=C_1(\alpha,\varepsilon,k)$ and $C_2=C_2(\alpha,\varepsilon,k)$ are positive constants. Moreover, for any $\alpha \in [0,1)$ and $n\geq 1$, one has
    \begin{equation}
      \label{upbdprobisolinear}
       \PP\left(I(n) \leq  \frac{(1-\alpha)n}{8}\right) \leq 2e^{-\frac{(1-\alpha)n}{18}}.
  \end{equation}
  \end{proposition}
\begin{remark}
When $\alpha$ is bounded away from $1$, one can get better estimates than (\ref{upbdprobisolinear}) at the cost of increasing the power of $1-\alpha$ in the exponent. For example, when $\alpha \leq 1/7$, the inequality (\ref{Mcdiarineoneside1}) implies that for any $n\geq 1$,
$$
\PP\left(I(n) \leq \frac{(1-\alpha)n}{8}\right) \leq  \PP\left(\frac{I(n)}{n} -  \frac{1-\alpha}{1+\alpha} \leq -\frac{3(1-\alpha)}{4}\right) \leq e^{-\frac{9(1-\alpha)^2n}{40}}  \leq e^{-\frac{27(1-\alpha)n}{140}}. 
$$ 
\end{remark}

By a slight abuse of notation, we denote by $\F_n$ the $\sigma$-algebra generated by the random forest $\F_n$, and in particular, $(\F_n)_{n\geq 1}$ is a filtration. We also use $\text{Bin}(n,p)$ for a random variable with binomial distribution. We shall use the following concentration inequalities, see e.g. \cite[Theorems 4.4 and 4.5]{MR3674428}: For any $\delta \in (0,1)$,
\begin{equation}
  \label{ineBinnp}
  \PP(\text{Bin}(n,p) \geq (1+\delta)np) \leq e^{-\frac{\delta^2 np}{3}}, \qquad
  \PP(\text{Bin}(n,p) \leq (1-\delta)np) \leq e^{-\frac{\delta^2 np}{2}}.
\end{equation}

\begin{proof}[Proof of Proposition \ref{Inestprop}]
We first consider the case $k=1$. A new vertex is isolated with probability $1-\alpha$, while an existing isolated vertex ceases to be isolated with probability $\alpha I(n)/n$. Thus,
$$
\EE(I(n+1)\mid\F_n)=\left(1-\frac{\alpha}{n}\right)I(n)+1-\alpha.
$$
Starting from $I(1)=1$, induction gives
\begin{equation}
    \label{EInequ}
\EE I(n)=\frac{(1-\alpha)n}{1+\alpha}
 +\frac{2\alpha}{1+\alpha}\prod_{j=1}^{n-1}\left(1-\frac{\alpha}{j}\right)
 \geq\frac{(1-\alpha)n}{1+\alpha}.
\end{equation}
The product is interpreted as $1$ when $n=1$. Its contribution to the expectation is at most $1$, so for any $\varepsilon>0$ it is at most $\varepsilon n/2$ for all sufficiently large $n$.

The random variable $I(n)$ is a function of the independent random variables $(\xi_j)_{2\leq j\leq n}$ and $(u_j)_{2\leq j\leq n}$. Changing one $\xi_j$ changes $I(n)$ by at most $2$. Changing one $u_j$ changes $I(n)$ by at most $1$: only the old and new parents can change their isolation status, and these changes have opposite signs. Consequently, McDiarmid's inequality and (\ref{EInequ}) give, for any $n\geq1$,
\begin{equation}
    \label{Mcdiarineoneside1}
\PP\left(\frac{I(n)}{n}-\frac{1-\alpha}{1+\alpha}\leq-\varepsilon\right)
\leq\PP\left(I(n)-\EE I(n)\leq-\varepsilon n\right)
\leq e^{-\frac{2\varepsilon^2 n}{5}}.
\end{equation}
Similarly, for all sufficiently large $n$,
$$
\PP\left(\frac{I(n)}{n}-\frac{1-\alpha}{1+\alpha}\geq\varepsilon\right)
\leq\PP\left(I(n)-\EE I(n)\geq\frac{\varepsilon n}{2}\right)
\leq e^{-\frac{\varepsilon^2 n}{10}}.
$$
This proves (\ref{expineNkn}) for $k=1$. For $k\geq2$, changing one $\xi_j$ or one $u_j$ changes $N_k(n)$ by at most $2$. Indeed, deleting an edge splits one cluster into two, while changing its parent transfers a subtree between two clusters. Thus, (\ref{expineNkn}) follows from McDiarmid's inequality and $\EE N_k(n)/n\to\theta_k$. Note that this also gives a second proof of the almost sure convergence in (\ref{limNkn}).

It remains to prove (\ref{upbdprobisolinear}). Write $q:=1-\alpha$, and let $V_n$ denote the number of vertices with no children in the random recursive tree before percolation. If $\mathscr R_n$ is the $\sigma$-algebra generated by $(u_j)_{2\leq j\leq n}$, then
$$
\EE(V_{n+1}-V_n\mid\mathscr R_n)=1-\frac{V_n}{n}.
$$
Indeed, the new vertex is a leaf, and its parent ceases to be a leaf precisely when it was a leaf before the attachment. Since $V_2=1$, taking expectations and using induction gives $\EE V_n=n/2$ for every $n\geq2$. Changing one parent choice can change the leaf status of only the old and new parents, with opposite signs, so it changes $V_n$ by at most $1$. McDiarmid's inequality therefore yields
$$
\PP(V_n<n/3)\leq\exp\left(-\frac{2(n/6)^2}{n-1}\right)\leq e^{-n/18},\qquad n\geq2.
$$
For $n\geq2$, every leaf has a parent edge. Given the original tree, these edges are deleted independently with probability $q$, and each deleted leaf is isolated in $\F_n$. Hence $I(n)$ is at least the number of deleted leaf edges, which has conditional distribution $\text{Bin}(V_n,q)$. On $\{V_n\geq n/3\}$, we have $qn/8\leq3qV_n/8$. Applying (\ref{ineBinnp}) with $\delta=5/8$ gives
\begin{equation}
    \label{leafisolationbound}
\begin{aligned}
\PP\left(I(n)\leq\frac{qn}{8}\right)
&\leq e^{-n/18}+e^{-25qn/384}\\
&\leq2e^{-qn/18}.
\end{aligned}
\end{equation}
The case $n=1$ is immediate. This completes the proof.
\end{proof}

\begin{corollary}
    \label{corNJn}
Let $J\subset\NN_+$ and write $N_J(n):=\sum_{k\in J}N_k(n)$. For any $\alpha\in(0,1)$ and $n\geq1$, one has
\begin{equation}
    \label{meanNJbound}
\left|\EE N_J(n)-n\sum_{k\in J}\theta_k\right|\leq2,
\end{equation}
where $(\theta_k)_{k\geq1}$ are defined in (\ref{limNkn}). Moreover, for any $n\geq2$ and $u>0$,
\begin{equation}
    \label{concentrationNJbound}
\PP\left(\left|N_J(n)-n\sum_{k\in J}\theta_k\right|\geq2+u\right)
\leq2\exp\left(-\frac{u^2}{4(n-1)}\right).
\end{equation}
In particular, for any $n\geq1$ and $\varepsilon>0$,
$$
\PP\left(\left|\frac{N_J(n)}n-\sum_{k\in J}\theta_k\right|\geq\varepsilon\right)
\leq C_0e^{-Cn},
$$
where $C_0=C_0(\varepsilon)$ and $C=C(\varepsilon)$ are positive constants. For the odd cluster sizes, this gives
$$
\PP\left(\left|\frac{\sum_{\ell\geq0}N_{2\ell+1}(n)}n-
\frac{1-\alpha}{2}\,{}_2F_1(1,\alpha^{-1};\alpha^{-1}+1;\tfrac12)\right|\geq\varepsilon\right)
\leq C_0e^{-Cn}.
$$
\end{corollary}
\begin{proof}
The Beta integral and monotone convergence give
$$
\sum_{k\geq1}\theta_k
=\frac{1-\alpha}{\alpha}\sum_{k\geq1}\int_0^1x^{1/\alpha}(1-x)^{k-1}\,dx
=1-\alpha.
$$
Also, the total number of clusters is $\sum_{k\geq1}N_k(n)=1+\sum_{j=2}^n(1-\xi_j)$, whose expectation is $(1-\alpha)n+\alpha$.

Set $e_n(k):=\EE N_k(n)-n\theta_k$ and $R_n:=\sum_{k\geq1}(e_n(k))^+$ where $(e_n(k))^+:=\max\{e_n(k),0\}$ is the positive part of $e_n(k)$. A cluster of size $k$ receives the next vertex with probability $\alpha k/n$, and a new cluster is born with probability $1-\alpha$. Together with
$$
(1+\alpha k)\theta_k=\alpha(k-1)\theta_{k-1}+(1-\alpha)\mathds{1}_{\{k=1\}},
$$
where $\theta_0:=0$, this yields
$$
e_{n+1}(k)=\left(1-\frac{\alpha k}{n}\right)e_n(k)
+\frac{\alpha(k-1)}n e_n(k-1),\qquad k\geq1,
$$
with $e_n(0):=0$. Let $n_0:=\max\{1,\lceil\alpha/(1-\alpha)\rceil\}$. For $n\geq n_0$, the coefficients in this recursion are nonnegative for $k\leq n+1$. Since $e_n(k)<0$ for $k>n$, summing positive parts only over $k\leq n+1$ gives
$$
R_{n+1}\leq\sum_{k=1}^{n}\left(1-\frac{\alpha k}{n}\right)(e_n(k))^+
+\sum_{k=1}^{n}\frac{\alpha k}{n}(e_n(k))^+=R_n.
$$
Since $n_0\leq1/(1-\alpha)$, for $1\leq n\leq n_0$ we have $R_n\leq\EE\sum_kN_k(n)=(1-\alpha)n+\alpha\leq1+\alpha<2$. Thus $R_n\leq2$ for every $n$. Since $\sum_{k\geq1}e_n(k)=\alpha\geq0$, the sum of the negative parts is at most $R_n$ as well. Hence $|\sum_{k\in J}e_n(k)|\leq R_n\leq2$, proving (\ref{meanNJbound}).

Changing one $\xi_j$ or one $u_j$ changes $N_J(n)$ by at most $2$. Indeed, changing $\xi_j$ splits one cluster into two or merges two clusters, while changing $u_j$ transfers a subtree between two clusters. McDiarmid's inequality, applied to these $2(n-1)$ independent variables, therefore gives
$$
\PP\left(|N_J(n)-\EE N_J(n)|\geq u\right)
\leq2\exp\left(-\frac{u^2}{4(n-1)}\right).
$$
Combining this with (\ref{meanNJbound}) proves (\ref{concentrationNJbound}). The exponential bound for deviations of size $\varepsilon n$ follows by taking $u=\varepsilon n-2$ when $n\geq4/\varepsilon$ and increasing $C_0$ to cover the remaining values of $n$.

For the odd cluster sizes, observe that
$$
\begin{aligned}
\sum_{\ell\geq0}\theta_{2\ell+1}
&=\frac{1-\alpha}{\alpha}\sum_{\ell\geq0}\int_0^1x^{1/\alpha}(1-x)^{2\ell}\,dx
=\frac{1-\alpha}{\alpha}\int_0^1\frac{x^{1/\alpha-1}}{2-x}\,dx\\
&=\frac{1-\alpha}{2\alpha}\sum_{m\geq0}\frac1{2^m}\int_0^1x^{1/\alpha+m-1}\,dx
=\frac{1-\alpha}{2\alpha}\sum_{m\geq0}\frac1{2^m(\alpha^{-1}+m)}\\
&=\frac{1-\alpha}{2}\sum_{m\geq0}
\frac{(1)_m(\alpha^{-1})_m}{(\alpha^{-1}+1)_m}\frac1{2^m m!}
=\frac{1-\alpha}{2}\,{}_2F_1(1,\alpha^{-1};\alpha^{-1}+1;\tfrac12).
\end{aligned}
$$
This also shows that ${}_2F_1(1,\alpha^{-1};\alpha^{-1}+1;1/2)=\sum_{m\geq0}2^{-m}(1+m\alpha)^{-1}$ is strictly decreasing in $\alpha\in(0,1)$, converges to $2$ as $\alpha\to0+$ and converges to $2\log2$ as $\alpha\to1-$. Moreover, since $\theta_3>0$, we have $\sum_{\ell\geq0}\theta_{2\ell+1}>\theta_1=(1-\alpha)/(1+\alpha)$, which gives ${}_2F_1(1,\alpha^{-1};\alpha^{-1}+1;1/2)>2/(1+\alpha)$ for $\alpha\in(0,1)$.
\end{proof}

Proposition \ref{Inestprop} and (\ref{limNkn}) describe the asymptotic behavior of $N_k(n)$ for fixed $k$. We will also be interested in the behavior of $N_k(n)$ when $k$ grows with $n$.  Baur and Bertoin \cite{MR3399834} proved that the size of the cluster rooted at a given vertex in $\F_n$, given that the vertex is a root, has growth rate $n^{\alpha}$. Gu and Yuan further proved the existence of the scaling limit of the size of the largest cluster in $\F_n$, after divided by $n^{\alpha}$, see \cite[Theorem 2]{Gu2026} and the remark after it. It is then natural to expect that if $n^{\alpha} \gg L$ where $L$ is some large positive number, then there are many clusters of sizes ranging from order $1$ to order $L$. The following result gives a quantitative description of this phenomenon. 

\begin{proposition}
    \label{cluster_sizewindowestlem}
Assume that $\alpha \in (0,1)$ and $L \geq 120$. There exist positive constants $C_1=C_1(\alpha)$ and $C_2=C_2(\alpha)$ such that if $n\geq L^{\frac{1}{\alpha}}$, then for any $k= 1,2,\dots, \lfloor L/40 \rfloor$, one has
$$
\PP\left(  \left|\left\{1\leq j \leq n: \frac{L}{96k} \leq |\CalC_{j,n}|<\frac{L}{2k}\right\}\right| \leq \frac{  C_2 k^{\frac{1}{\alpha}} n}{L^{\frac{1}{\alpha}}} \right)  \leq 4\exp\left(- \frac{ C_1  k^{\frac{1}{\alpha}}n}{L^{\frac{1}{\alpha}}}\right).
$$
In the case $\alpha=1/2$, if $L/\sqrt{\log L} \geq 120$, then for any $n\geq  L^{2}/\log L$ and any integer $k \in [\sqrt{\log L}, \lfloor L/40 \rfloor]$, one has,  
$$
\PP\left(  \left|\left\{1\leq j \leq n: \frac{L}{96k} \leq |\CalC_{j,n}|<\frac{L}{k}\right\}\right| \leq \frac{  C_2 k^{2} n}{4 L^{2}} \right)  \leq 4\exp\left(- \frac{ C_1  k^{2}n}{4 L^{2}}\right).
$$
\end{proposition}

Throughout the rest of this section, we assume that $\alpha \in (0,1)$ and $L\geq 120$ and $n\geq L^{\frac{1}{\alpha}}$. For $k=1,2,\dots, \lfloor L/40 \rfloor$, we write 
$$t(k):=\left\lceil \left(\frac{20k}{L}\right)^{\frac{1}{\alpha}} n \right\rceil \leq \frac{n}{2}+1.$$
Note that $t(k)\geq 20 k^{\frac{1}{\alpha}} \geq 20$. Recall that $\mathscr{I}_{t(k)}$ denotes the set of isolated vertices in $\F_{t(k)}$, which has cardinality $I(t(k))$. By Proposition \ref{Inestprop},
\begin{equation}
    \label{probbdE1k}
    \PP\left(E_{1}(k)^c \right)\leq 2e^{-(1-\alpha)t(k)/18}, \text{ where } E_{1}(k):=\{I(t(k)) > (1-\alpha)t(k) /8 \}.
\end{equation}
We will prove Proposition \ref{cluster_sizewindowestlem} by showing that given that $E_{1}(k)$ holds, then in $\F_n$, each cluster rooted at a vertex in $\mathscr{I}_{t(k)}$ has size between $L/(96k)$ and $L/(2k)$ with positive probability bounded away from $0$, see Lemma \ref{cluster_size_meanineq}. We then establish in Lemma \ref{lemnegacor} the negative correlation between these clusters, which enables us to show that with high probability, at least one sixteenth of these clusters have sizes in the window $[L/(96k), L/(2k))$. See Figure \ref{growFn} for an illustration of the proof strategy.

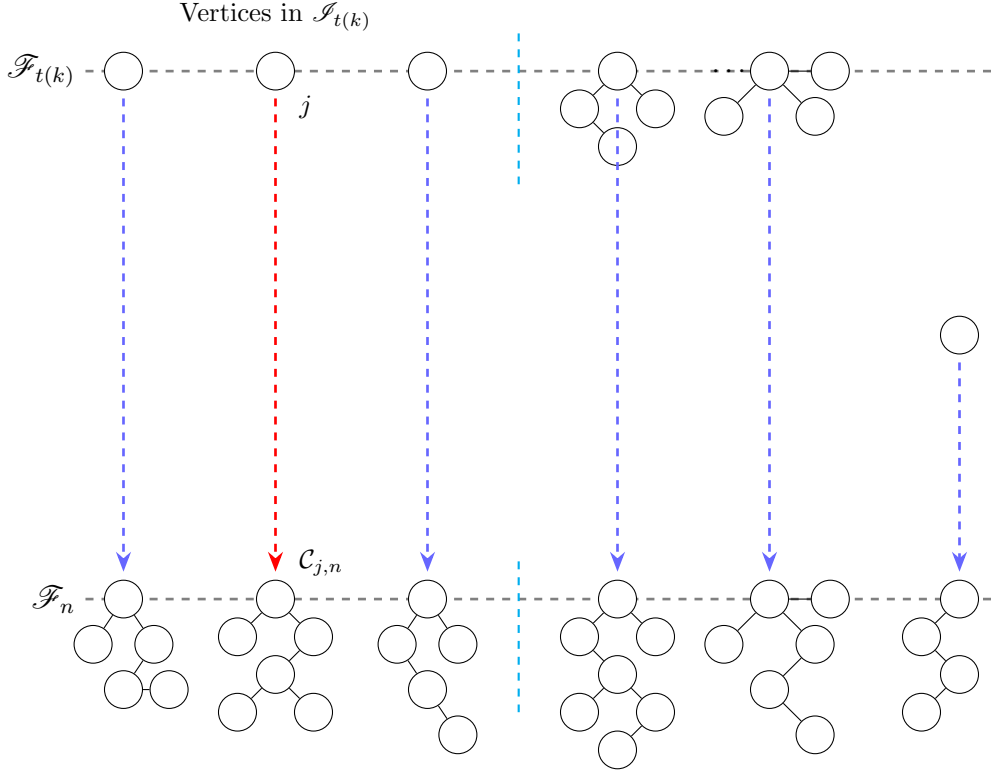
\begin{figure}
  \begin{tikzpicture}[
    node distance=0.5cm,
    timepoint/.style={font=\large\bfseries, align=center},
    component/.style={circle, draw, minimum size=0.5cm, inner sep=0pt, fill=white},
    grow arrow/.style={->, >=Stealth, shorten >=3pt, shorten <=3pt, line width=1pt, blue!60, dashed},
    time line/.style={dashed, line width=1pt, gray}
]

\draw[time line] (-0.5,0) -- (11.5,0) node[midway, below=0.2cm] {};
\draw[time line] (-0.5,-4.2) -- (11.5,-4.2) node[midway, below=0.2cm] {};

\node[timepoint, anchor=east] (tk) at (-0.5,0) {$\F_{t(k)}$};
\node[timepoint, anchor=east] (tn) at (-0.5,-4.2) {$\F_n$};

\node[component] (a1) at (0,0) {};
\node[component] (a2) at (2,0) {};
\node[component] (a3) at (4,0) {};

\node[below right=0cm of a2] {\small $j$};

\node at (8,0) {$\cdots$};
\node[above=0.15cm of a2, align=center, font=\small] {Vertices in $\mathscr{I}_{t(k)}$};

\node[component] (a4) at (11,-2.1) {};

\begin{scope}[xshift=6.5cm, yshift=0cm]
    \node[component] (t1) at (0,0) {};
    \node[component] (t2) at (0.5,-0.5) {};
    \node[component] (t3) at (-0.5,-0.5) {};
    \node[component] (t4) at (0,-1) {};
    \draw (t1) -- (t2) (t1) -- (t3) (t3) -- (t4);
\end{scope}

\begin{scope}[xshift=8.5cm, yshift=0cm]
    \node[component] (u1) at (0,0) {};
    \node[component] (u2) at (0.6,-0.6) {};
    \node[component] (u3) at (-0.6,-0.6) {};
    \node[component] (u4) at (0.8,0) {};
    \draw (u1) -- (u2) (u1) -- (u3) (u1) -- (u4);
\end{scope}

\draw[cyan, dashed, thick] (5.2, 0.5) -- (5.2, -1.5);


\begin{scope}[yshift=-4.2cm, xshift=0cm]
    \node[component] (d1) at (0,0) {};
    \node[component] (d2) at (-0.4,-0.6) {};
    \node[component] (d3) at (0.4,-0.6) {};
    \node[component] (d4) at (0,-1.2) {};
    \node[component] (d5) at (0.6,-1.2) {};
    \draw (d1) -- (d2) (d1) -- (d3) (d3) -- (d4) (d4) -- (d5);
\end{scope}

\begin{scope}[yshift=-4.2cm, xshift=2cm]
    \node[component] (e1) at (0,0) {};
    \node[component] (e2) at (0.5,-0.5) {};
    \node[component] (e3) at (-0.5,-0.5) {};
    \node[component] (e4) at (0,-1.0) {};
    \node[component] (e5) at (0.5,-1.5) {};
    \node[component] (e6) at (-0.5,-1.5) {};
    \draw (e1) -- (e2) (e1) -- (e3) (e2) -- (e4) (e4) -- (e5) (e4) -- (e6);
    \node[above right=0cm of e1, font=\small] {$\CalC_{j,n}$};
\end{scope}

\begin{scope}[yshift=-4.2cm, xshift=4cm]
    \node[component] (f1) at (0,0) {};
    \node[component] (f2) at (0.4,-0.6) {};
    \node[component] (f3) at (-0.4,-0.6) {};
    \node[component] (f4) at (0,-1.2) {};
    \node[component] (f5) at (0.4,-1.8) {};
    \draw (f1) -- (f2) (f1) -- (f3) (f3) -- (f4) (f4) -- (f5);
\end{scope}

\begin{scope}[yshift=-4.2cm, xshift=6.5cm]
    \node[component] (g1) at (0,0) {};
    \node[component] (g2) at (0.5,-0.5) {};
    \node[component] (g3) at (-0.5,-0.5) {};
    \node[component] (g4) at (0,-1.0) {};
    \node[component] (g5) at (0.5,-1.5) {};
    \node[component] (g6) at (-0.5,-1.5) {};
    \node[component] (g7) at (0,-2.0) {};
    \draw (g1) -- (g2) (g1) -- (g3) (g3) -- (g4) (g4) -- (g5) (g4) -- (g6) (g5) -- (g7);
\end{scope}

\begin{scope}[yshift=-4.2cm, xshift=8.5cm]
    \node[component] (h1) at (0,0) {};
    \node[component] (h2) at (0.6,-0.6) {};
    \node[component] (h3) at (-0.6,-0.6) {};
    \node[component] (h4) at (0.8,0) {};
    \node[component] (h5) at (0,-1.2) {};
    \node[component] (h6) at (0.6,-1.8) {};
    \draw (h1) -- (h2) (h1) -- (h3) (h1) -- (h4) (h2) -- (h5) (h5) -- (h6);
\end{scope}

\begin{scope}[yshift=-4.2cm, xshift=11cm]
    \node[component] (i1) at (0,0) {};
    \node[component] (i2) at (-0.5,-0.5) {};
    \node[component] (i3) at (0,-1.0) {};
    \node[component] (i4) at (-0.5,-1.5) {};
    \draw (i1) -- (i2) (i2) -- (i3) (i3) -- (i4);
\end{scope}

\draw[cyan, dashed, thick] (5.2, -3.7) -- (5.2, -5.7);

\draw[grow arrow] (a1) -- (d1);
\draw[grow arrow] (a3) -- (f1);
\draw[grow arrow] (t1) -- (g1);
\draw[grow arrow] (u1) -- (h1);
\draw[grow arrow] (a4) -- (i1);

\draw[->, >=Stealth, shorten >=3pt, shorten <=3pt, line width=1pt, red, dashed] (a2.south) -- (e1.north);

\end{tikzpicture}
  \caption{Growth of the forest from time $t(k)$ to time $n$. The three isolated vertices before the cyan dashed line in $\F_{t(k)}$ grow into the three clusters below them in $\F_n$. The other clusters grow from non-isolated components or from vertices added after time $t(k)$.}
  \label{growFn}
\end{figure}

\begin{lemma}
  \label{cluster_size_meanineq}
Fix $k\in\{1,\ldots,\lfloor L/40\rfloor\}$ and condition on $\F_{t(k)}$. Assume that $\mathscr I_{t(k)}$ is nonempty. Then, for every $j\in\mathscr I_{t(k)}$,
$$
\EE\left(|\CalC_{j,n}|\mid\F_{t(k)}\right)=\frac{a_n}{a_{t(k)}},
\qquad a_m:=\prod_{\ell=1}^{m-1}\left(1+\frac\alpha\ell\right),\quad m\geq1.
$$
If $n\geq L^{1/\alpha}$, then
$$
\PP\left(|\CalC_{j,n}|\geq\frac{L}{2k}\mid\F_{t(k)}\right)<\frac18,
\qquad
\PP\left(|\CalC_{j,n}|\geq\frac{L}{96k}\mid\F_{t(k)}\right)\geq\frac9{32}.
$$
\end{lemma}
\begin{proof}[Proof of Lemma \ref{cluster_size_meanineq}]
Throughout the proof, all expectations and probabilities are conditional on $\F_{t(k)}$. Fix $j\in\mathscr I_{t(k)}$ and let
$$
M_k(m):=\frac{a_{t(k)}|\CalC_{j,m}|}{a_m},\qquad m\geq t(k).
$$
Since
$$
\EE\left(|\CalC_{j,m+1}|-|\CalC_{j,m}|\mid\F_m\right)
=\frac{\alpha|\CalC_{j,m}|}{m},
$$
the process $(M_k(m))_{m\geq t(k)}$ is a martingale starting from $1$, which proves the first assertion.

For any $1\leq t\leq n$, we have
$$
\left(\frac nt\right)^\alpha\leq\frac{a_n}{a_t}
\leq e^{\alpha/t}\left(\frac nt\right)^\alpha.
$$
The lower bound follows by multiplying $(1+1/\ell)^\alpha\leq1+\alpha/\ell$. For the upper bound, use $\log(1+\alpha/\ell)\leq\alpha/\ell$ and
$\sum_{\ell=t}^{n-1}\ell^{-1}\leq\log(n/t)+1/t$.
Write $x:=(20k/L)^{1/\alpha}n$ and $t=t(k)=\lceil x\rceil$. If $n\geq L^{1/\alpha}$, then $x\geq(20k)^{1/\alpha}\geq20$, so $x\leq t\leq21x/20$. Consequently, using $\alpha < 1$, one has,
$$
\frac{L}{21k}\leq\EE|\CalC_{j,n}|\leq e^{1/20}\frac{L}{20k},
\qquad\text{and hence}\qquad
\frac{L}{24k}<\EE|\CalC_{j,n}|<\frac{L}{16k}.
$$
Markov's inequality therefore gives
$$
\PP\left(|\CalC_{j,n}|\geq\frac{L}{2k}\right)
\leq\frac{2k\EE|\CalC_{j,n}|}{L}<\frac18.
$$
Similarly, since
$$
\EE\left(|\CalC_{j,m+1}|^2-|\CalC_{j,m}|^2\mid\F_m\right)
=\frac{2\alpha|\CalC_{j,m}|^2}{m}+\frac{\alpha|\CalC_{j,m}|}{m}
$$
and $\EE|\CalC_{j,m}|=a_m/a_{t(k)}$, we obtain
$$
\EE|\CalC_{j,m+1}|^2+\frac{a_{m+1}}{a_{t(k)}}
=\left(1+\frac{2\alpha}{m}\right)
\left(\EE|\CalC_{j,m}|^2+\frac{a_m}{a_{t(k)}}\right).
$$
Iterating this identity from $m=t(k)$ and using $1+2z\leq(1+z)^2$ gives
$$
\begin{aligned}
\EE|\CalC_{j,n}|^2
&=2\prod_{\ell=t(k)}^{n-1}\left(1+\frac{2\alpha}{\ell}\right)-\frac{a_n}{a_{t(k)}}\\
&\leq2\prod_{\ell=t(k)}^{n-1}\left(1+\frac\alpha\ell\right)^2
=2\left(\EE|\CalC_{j,n}|\right)^2.
\end{aligned}
$$
Finally, $L/(96k)<\EE|\CalC_{j,n}|/4$, so the Paley--Zygmund inequality yields
$$
\PP\left(|\CalC_{j,n}|\geq\frac{L}{96k}\right)
\geq\PP\left(|\CalC_{j,n}|\geq\frac{\EE|\CalC_{j,n}|}{4}\right)
\geq\left(1-\frac14\right)^2
\frac{(\EE|\CalC_{j,n}|)^2}{\EE|\CalC_{j,n}|^2}
\geq\frac9{32}.
$$
\end{proof}

Given a non-empty finite index set $\tilde{J}$, we say that the random variables $\{Y_j\}_{j \in \tilde{J}}$  taking values in $\{0,1\}$ are {\bf negatively correlated} if for any non-empty subset $J \subset \tilde{J}$, one has 
$$
\PP\left(\bigcap_{j\in J}\{Y_j=1\}\right) \leq \prod_{j\in J} \PP(Y_j=1).
$$

\begin{lemma}
\label{lemnegacor}
Let $(\CalC_{j,m})_{j \in \tilde{J}}$ denote the non-empty clusters in $\F_m$ where $m\geq 2$. Then, given $\F_m$, for any $K>0$ and any $n \geq m$, the indicator functions  $\{\mathds{1}_{\{|\CalC_{j,n}| \geq K\}}\}_{j\in \tilde{J}}$ are negatively correlated, and $\{\mathds{1}_{\{|\CalC_{j,n}| < K\}}\}_{j\in \tilde{J}}$ are also negatively correlated.
\end{lemma}
\begin{proof}
Again, we omit ``conditionally on $\F_m$'' for simplicity of notation. Let $J\subset\tilde J$ with $|J|\geq2$, fix $j_*\in J$, and write
$$
\CalC_t:=\CalC_{j_*,t},\qquad
E_J:=\bigcap_{j\in J\setminus\{j_*\}}\{|\CalC_{j,n}|\geq K\}.
$$
We may assume that $\PP(E_J)>0$. We shall prove the following stochastic domination: for every $t\in\{m,\ldots,n\}$ and every real $r$,
\begin{equation}
    \label{negacorcondprob}
\PP(|\CalC_t|\geq r\mid E_J)\leq\PP(|\CalC_t|\geq r).
\end{equation}
Taking $t=n$ and $r=K$, and then using induction on $|J|$, proves the first assertion.

We first show that, whenever the conditional probability is defined,
\begin{equation}
    \label{bayestbp}
\PP(|\CalC_{\ell+1}|=k+1\mid E_J,|\CalC_\ell|=k)
\leq\frac{\alpha k}{\ell},\qquad m\leq\ell<n.
\end{equation}
Fix a realization of the forest up to time $\ell$ with $|\CalC_\ell|=k$, and keep all parent choices and percolation variables after time $\ell+1$ fixed. Replacing the edge at time $\ell+1$ by a retained edge to $\CalC_\ell$ can only decrease the final sizes of the other clusters indexed by $J$. Indeed, it moves the new vertex and its connected future descendants to $\CalC$, and leaves all other vertices in their previous clusters. All attachments to vertices in $\CalC_\ell$ have the same effect on these cluster sizes. It follows that $E_J$ is no more likely when $|\CalC_{\ell+1}|=k+1$ than without this conditioning. The probability of this attachment is $\alpha k/\ell$ for every past forest with $|\CalC_\ell|=k$, so averaging over such forests preserves the comparison. Thus,
$$
\PP(E_J\mid |\CalC_\ell|=k,|\CalC_{\ell+1}|=k+1)
\leq\PP(E_J\mid |\CalC_\ell|=k),
$$
and (\ref{bayestbp}) follows from Bayes' formula.

We now prove (\ref{negacorcondprob}) by induction on $t$. It is immediate at $t=m$, since $|\CalC_m|$ is fixed. Suppose that it holds at $t=\ell$ for every $r$. We can then couple random variables $X$ and $Y$ such that
$$
X\sim\mathcal L(|\CalC_\ell|),\qquad
Y\sim\mathcal L(|\CalC_\ell|\mid E_J),\qquad Y\leq X.
$$
Let $u$ be uniform on $(0,1)$ and independent of $(X,Y)$. Set
$$
\Delta_X=\mathds{1}_{\{u\leq\alpha X/\ell\}},\qquad
\Delta_Y=\mathds{1}_{\{u\leq p_\ell(Y)\}},\quad
p_\ell(k):=\PP(|\CalC_{\ell+1}|=k+1\mid E_J,|\CalC_\ell|=k).
$$
Then $X+\Delta_X$ has the law of $|\CalC_{\ell+1}|$, and $Y+\Delta_Y$ has its conditional law given $E_J$. By (\ref{bayestbp}), $p_\ell(Y)\leq\alpha Y/\ell\leq\alpha X/\ell$, so $Y+\Delta_Y\leq X+\Delta_X$. This proves the induction step for every $r$.

For the second assertion, replace $E_J$ by
$$
E_J^-:=\bigcap_{j\in J\setminus\{j_*\}}\{|\CalC_{j,n}|<K\}.
$$
The same edge replacement can only increase the indicator of $E_J^-$. Bayes' formula therefore gives $p_\ell^-(k)\geq\alpha k/\ell$, with the analogous definition of $p_\ell^-$. The coupling above, with the inequalities reversed, shows that $|\CalC_t|$ conditional on $E_J^-$ stochastically dominates its unconditional law. In particular,
$$
\PP(|\CalC_n|<K\mid E_J^-)\leq\PP(|\CalC_n|<K).
$$
Induction on $|J|$ proves the second assertion.
\end{proof}

\begin{proof}[Proof of Proposition \ref{cluster_sizewindowestlem}]
Given $\F_{t(k)}$, on the event $E_1(k)$ defined in (\ref{probbdE1k}), write $I:=I(t(k))$. Lemmas \ref{cluster_size_meanineq} and \ref{lemnegacor} show that the indicators of $\{|\CalC_{j,n}|\geq L/(2k)\}$, $j\in\mathscr I_{t(k)}$, are negatively correlated and have success probabilities at most $1/8$. The Chernoff--Hoeffding bound for negatively correlated indicators (see e.g. \cite[Theorem 3.4]{MR1438520}), with $\varepsilon=1/2$, gives
$$
\PP\left(\sum_{j\in\mathscr I_{t(k)}}\mathds{1}_{\{|\CalC_{j,n}|\geq L/(2k)\}}\geq\frac{3I}{16}\,\middle|\,\F_{t(k)}\right)
\mathds{1}_{E_1(k)}\leq e^{-I/96}\mathds{1}_{E_1(k)}.
$$
Similarly, the indicators of $\{|\CalC_{j,n}|<L/(96k)\}$ are negatively correlated and have success probabilities at most $23/32$. Since $3/4=(1+1/23)23/32$, the same bound with $\varepsilon=1/23$ yields
$$
\PP\left(\sum_{j\in\mathscr I_{t(k)}}\mathds{1}_{\{|\CalC_{j,n}|<L/(96k)\}}\geq\frac{3I}{4}\,\middle|\,\F_{t(k)}\right)
\mathds{1}_{E_1(k)}\leq e^{-I/2208}\mathds{1}_{E_1(k)}.
$$
Outside these two exceptional events, and on $E_1(k)$, the number of clusters in the desired window is greater than
$$
I-\frac{3I}{16}-\frac{3I}{4}=\frac{I}{16}
>\frac{(1-\alpha)t(k)}{128}
\geq\frac{(1-\alpha)20^{1/\alpha}}{128}
\frac{k^{1/\alpha}n}{L^{1/\alpha}}.
$$
We can therefore take $C_2=(1-\alpha)20^{1/\alpha}/128$. Combining the two bounds with (\ref{probbdE1k}) gives a total exceptional probability at most
$$
2e^{-(1-\alpha)t(k)/18}+2e^{-(1-\alpha)t(k)/17664}
\leq4\exp\left(-C_1\frac{k^{1/\alpha}n}{L^{1/\alpha}}\right),
$$
where $C_1:=(1-\alpha)20^{1/\alpha}/17664$. This proves the first assertion.

Now assume that $\alpha=1/2$, $L/\sqrt{\log L} \geq 120$, and $n\geq  L^{2}/\log L$. We let $\tilde{L}:=L/\sqrt{\log L}$. For any $k \in [\sqrt{\log L}, \lfloor L/40 \rfloor]$, we can find an integer $\tilde{k} \in [1, \lfloor \tilde{L}/40 \rfloor]$ such that $\tilde{k}\sqrt{\log L}\leq k < (\tilde{k}+1)\sqrt{\log L}$ $(\leq 2\tilde{k} \sqrt{\log L})$. Then noting that $\tilde{L}/(2\tilde{k}) \leq L/k \leq \tilde{L}/\tilde{k}$ and using the first desired inequality (for $\tilde{k}$ and $\tilde{L}$), we have 
$$ 
\begin{aligned}
   &\quad \ \PP\left(  \left|\left\{1\leq j \leq n: \frac{L}{96k} \leq |\CalC_{j,n}|<\frac{L}{k}\right\}\right| \leq \frac{  C_2 k^{2} n}{4 L^{2}} \right)  \\
   &\leq \PP\left(  \left|\left\{1\leq j \leq n: \frac{\tilde{L}}{96\tilde{k}} \leq |\CalC_{j,n}|<\frac{\tilde{L}}{2\tilde{k}}\right\}\right| \leq \frac{  C_2 \tilde{k}^{2} n}{ \tilde{L}^{2}} \right)  \\
   &\leq 4\exp\left(- \frac{  C_1 \tilde{k}^{2} n}{ \tilde{L}^{2}} \right) \leq 4\exp\left(- \frac{  C_1 k^{2} n}{4 L^{2}} \right),
\end{aligned}
$$
which completes the proof.
\end{proof}

\section{Proof of main results}
\label{proofsec}

Throughout this section, we let the random forests $(\F_n)_{n\geq 1}$ and the i.i.d. $\mu$-distributed random variables $(g_n)_{n\geq 1}$ be as in Section \ref{secsrrwrrt}, and let $\left|\CalC_{j, n}\right|$ denote the size of the cluster in the forest $\F_n$ rooted at $j \leq n$. Let $I(n)$ denote the size of $\mathscr{I}_n$, which is the set of isolated vertices in $\F_n$. Let $S$ be an SRRW with step distribution $\mu$ and parameter $\alpha$ as in Proposition \ref{consSRRWRRT}.

\subsection{Proof of Proposition \ref{propabelcompare}: Counting free steps}
\label{proofpropabelcompare}

When $\mu$ is a class function, the transition matrix $P_{\mu}$ in (\ref{defPkellsec}) commutes with transition matrices of the form $P^{(g)}$ for $g \in \Gamma$ (see (\ref{classinterchmug}) below). Together with knowledge of $t_{\text{mix}}^{(0)}(\varepsilon)$, this commutativity allows us to control $t_{\text{mix}}^{(\alpha)}(\varepsilon)$ provided that there are sufficiently many free steps (i.e., isolated vertices).
\begin{proof}[Proof of Proposition \ref{propabelcompare}]
(i). Recall that given $\F_n, (g_j)_{j \in [n]\backslash \mathscr{I}_n}$, the conditional distribution of $S_n$ is given by 
$$
\PP( S_n=\cdot \mid \F_n, (g_j)_{j \in [n]\backslash \mathscr{I}_n})=\delta_{e_G}P_1P_2\cdots P_n,
$$
where each $P_j$ is either $P_{\mu}$ or $P^{(g)}$ for some $g\in \Gamma$. If $\mu$ is a class function, then for any $g$,
\begin{equation}
    \label{classinterchmug}
    P_{\mu} P^{(g)}=P^{(g)} P_{\mu},
\end{equation}
since for any $x,y\in G$, 
$$
\begin{aligned}
  \sum_{z\in G} P_{\mu}(x,z) P^{(g)}(z,y)&=P_{\mu}(x,y\cdot g^{-1})=\mu(x^{-1}\cdot y \cdot g^{-1})=\mu(g^{-1} \cdot x^{-1}\cdot y )\\
  &=P_{\mu}(x\cdot g,y) =\sum_{z\in G} P^{(g)}(x,z) P_{\mu}(z,y).
\end{aligned}
$$
If $m_1<m_2<\dots$ denote the non-isolated vertices in $\F_n$, then using the commutativity (\ref{classinterchmug}), we can write 
$$\PP( S_n=\cdot \mid \F_n, (g_j)_{j \in [n]\backslash \mathscr{I}_n})=\delta_{e_G}  P_{m_1}P_{m_2}\cdots  P_{m_{n-I(n)}}P_{\mu}^{I(n)}.$$
By the definition of $t^{(0)}_{\mathrm{mix}}(\varepsilon/2)$,
$$
 \|\PP(S_{n} = \cdot \mid \F_n, (g_j)_{j \in [n]\backslash \mathscr{I}_n}) -  U\|_{\mathrm{TV}}\mathds{1}_{\{I(n)\geq t^{(0)}_{\mathrm{mix}}(\varepsilon/2)\}}\leq \frac{\varepsilon}{2}.$$
Therefore, Proposition \ref{propcondSnU} shows that
 \begin{equation}
     \label{equtvdineIn}
     \|\PP(S_{n}=\cdot)- U\|_{\mathrm{TV}}  \leq \frac{\varepsilon}{2}+\PP\left(I(n)< t^{(0)}_{\mathrm{mix}}\left(\frac{\varepsilon}{2}\right)\right).
 \end{equation}
If
   $$
n\geq  \frac{1}{1-\alpha} \max\left\{8t^{(0)}_{\mathrm{mix}}\left(\frac{\varepsilon}{2}\right),18\log \left(\frac{4}{\varepsilon}\right)\right\},
   $$
   then by the second inequality in Proposition \ref{Inestprop}, 
   $$
\PP\left(I(n)< t^{(0)}_{\mathrm{mix}}\left(\frac{\varepsilon}{2}\right)\right)\leq \PP\left(I(n) < \frac{(1-\alpha)n}{8}\right) \leq 2e^{-\frac{(1-\alpha)n}{18}} \leq \frac{\varepsilon}{2},
   $$
   which completes the proof. 

   (ii). Assume that $S$ is an SRRW on the group $G_m$ with step distribution $\mu_m$ and reinforcement parameter $\alpha \in [0,1)$. Then, (\ref{equtvdineIn}) becomes 
   $$
    \|\PP(S_{n}=\cdot)- U\|_{\mathrm{TV}}  \leq \frac{\varepsilon}{2}+\PP\left(I(n)< t^{(0,G_m,\mu_m)}_{\mathrm{mix}}\left(\frac{\varepsilon}{2}\right)\right).
   $$
Fix $\varepsilon \in (0,1)$ and $\delta \in (0,(1-\alpha)/(1+\alpha))$, by our assumption that $t^{(0, G_n,\mu_n)}_{\mathrm{mix}}\left(\varepsilon/2\right) \to \infty$, we can find $m(\varepsilon,\delta)>0$ such that for all $m \geq m(\varepsilon,\delta)$,
$$
t^{(0,G_m,\mu_m)}_{\mathrm{mix}}\left(\frac{\varepsilon}{2}\right) > \frac{5}{2\delta^2} \log \left(\frac{2}{\varepsilon}\right).
$$
Thus, for any $m\geq m(\varepsilon,\delta)$, if
   $$
\left(\frac{1-\alpha}{1+\alpha}- \delta \right) n \geq  t^{(0,G_m,\mu_m)}_{\mathrm{mix}}\left(\frac{\varepsilon}{2}\right),
   $$
then by (\ref{Mcdiarineoneside1}), 
   $$
\PP\left(I(n)< t^{(0)}_{\mathrm{mix}}\left(\frac{\varepsilon}{2}\right)\right)\leq \PP\left(\frac{I(n)}{n} - \frac{1-\alpha}{1+\alpha} \leq -\delta\right) \leq e^{-\frac{2\delta^2 n}{5}}<\frac{\varepsilon}{2}.
   $$
This shows that for any $m\geq m(\varepsilon,\delta)$,
$$
t^{(\alpha,G_m,\mu_m)}_{\mathrm{mix}}(\varepsilon) \leq 1+ \left(\frac{1-\alpha}{1+\alpha}- \delta \right)^{-1}t^{(0,G_m,\mu_m)}_{\mathrm{mix}}\left(\frac{\varepsilon}{2}\right),
$$
which proves the desired result by letting $m \to \infty$ since we can choose $\delta$ to be arbitrarily small.
\end{proof}

\subsection{Proof of Proposition \ref{conexpthm}: Counting consecutive free steps}
\label{contractkersec}

  Since $P_{\mu}$ is assumed to be irreducible and aperiodic, there exists a positive integer $m_*$ and a positive number $\varepsilon_*$ such that $P_{\mu}^{m_*}(x,y)\geq \varepsilon_*$ for all $x\in G,y\in G$ (in particular, $\varepsilon_*|G|\leq 1$). It is known that $P_{\mu}^{m_*}$ is a strict contraction of the probability space on $G$ relative to total variation distance, see e.g. \cite[Lemma 3.25]{MR2548569}: For any two probability measures $\nu_1,\nu_2$ on $G$, 
$$
  \|\nu_1P_{\mu}^{m_*}- \nu_2P_{\mu}^{m_*}\|_{\mathrm{TV}} \leq (1-|G|\varepsilon_*) \|\nu_1- \nu_2\|_{\mathrm{TV}}.
$$
and in particular, 
\begin{equation}
  \label{contraPm}
 \|\nu_1P_{\mu}^{m_*}- U\|_{\mathrm{TV}}=\|\nu_1P_{\mu}^{m_*}- UP_{\mu}^{m_*}\|_{\mathrm{TV}} \leq (1-|G|\varepsilon_*) \|\nu_1- U\|_{\mathrm{TV}}.
 \end{equation}
 In light of Proposition \ref{propcondSnU}, this observation (\ref{contraPm}) motivates us to count how many disjoint copies of $P_{\mu}^{m_*}$ appear in the product $\prod_{k=1}^nP_k$ (by (\ref{defPkellsec}), this product is the conditional transition matrix $P_{0,n}$). Equivalently, we are interested in the number of disjoint blocks of length $m_*$ contained in $\mathscr{I}_n$. For $k\geq 1$, define
$$
I^{(m_*)}(k m_*):= \sum_{j=1}^k \mathds{1}_{\{ \{m_*(j-1)+1,m_*(j-1)+2,\dots,m_* j\} \subset  \mathscr{I}_{k m_*} \}},
$$
which counts the blocks of the form $\{m_*(j-1)+1,\dots,m_* j\}$ whose every vertex is isolated in $\F_{k m_*}$. The following Lemma \ref{mstarblolem} is the block analogue of (\ref{upbdprobisolinear}).
\begin{lemma}
\label{mstarblolem}
  There exist positive constants $C_1$ and $C_2$, depending only on $m_*$, such that for any $\alpha \in [0,1)$ and any $k>m_*^2+1$, one has
$$
\PP(I^{(m_*)}(k m_*) \leq C_1 (1-\alpha)^{m_*} k) \leq 2e^{-C_2 (1-\alpha)^{m_*} k }.
$$
\end{lemma}
\begin{proof}
Write $n=km_*$ and $p=(1-\alpha)^{m_*}$. Consider the random recursive tree before percolation, and let $V$ count the blocks
$$
B_j:=\{(j-1)m_*+1,\dots,jm_*\},\qquad \lceil k/2\rceil<j\leq k,
$$
whose vertices are all leaves of this tree. For such a block, write $a=(j-1)m_*+1$ and $r_i=|B_j\cap[i-1]|$. The independence of the parent choices gives
$$
\PP(B_j\text{ consists of leaves})
=\prod_{i=a+1}^n\left(1-\frac{r_i}{i-1}\right)\geq e^{-2m_*}.
$$
Indeed, $a-1\geq n/2$, so 
$$\frac{r_i}{i-1} \leq \frac{ m_*}{a} = \frac{2}{k} \leq \frac{2}{3}\quad \text{ and } \quad \sum_{i=a+1}^n \frac{r_i}{i-1} \leq \frac{m_*(n-a)}{a} \leq m_*.$$ 
We then use
$\log(1-x)\geq-2x$ for $0\leq x\leq2/3$. There are $\lfloor k/2\rfloor\geq k/3$ such blocks. Hence, with $b=e^{-2m_*}/3$,
$$
\EE V\geq bk.
$$
Changing one parent choice can make at most one block cease to consist of leaves and at most one other block start to consist of leaves. These changes have opposite signs, so $V$ changes by at most one. McDiarmid's inequality therefore gives
$$
\PP(V<bk/2)\leq\exp\left(-\frac{b^2k^2}{2(n-1)}\right)
\leq\exp\left(-\frac{b^2k}{2m_*}\right).
$$
Given the tree, let $Z$ count those $V$ blocks for which all $m_*$ parent edges are deleted. None of these blocks contains the root, and their parent edges are distinct. Thus $Z$ has conditional distribution $\operatorname{Bin}(V,p)$ and $I^{(m_*)}(km_*)\geq Z$. On the event $V\geq bk/2$, the conditional mean of $Z$ is at least $bpk/2$, and applying (\ref{ineBinnp}) with $\delta=1/2$ gives
$$
\PP(Z\leq bpk/4\mid\text{the tree})\leq e^{-pV/8}\leq e^{-bpk/16}.
$$
Consequently,
$$
\PP\left(I^{(m_*)}(km_*)\leq\frac b4pk\right)
\leq e^{-b^2k/(2m_*)}+e^{-bpk/16}
\leq 2e^{-C_2pk},
$$
where $C_2=\min\{b^2/(2m_*),b/16\}$. This proves the lemma with $C_1=b/4$.
\end{proof}

We are now ready to prove Proposition \ref{conexpthm}.

\begin{proof}[Proof of Proposition \ref{conexpthm}] 
 we first assume that $n>m_*(m_*^2+1)$ such that $k m_* \leq n < (k+1)m_*$ for some integer $k\geq m_*^2+1$. Since $U$ is stationary for each $P_{k}$, we see that 
$$\| \delta_{e_G}\prod_{k=1}^{\ell}P_k-U\prod_{k=1}^{\ell}P_k  \|_{\mathrm{TV}}$$ 
is non-increasing in $\ell \in [n]$. Observe that the number of disjoint blocks of length $m_*$ contained in $\mathscr{I}_n$ is at least $I^{(m_*)}((k+1) m_*)-1$, the contraction inequality (\ref{contraPm}) shows that (we may assume that $\varepsilon_*\leq 1/(2|G|)$)
$$
\| \delta_{e_G}\prod_{k=1}^{n}P_k-U\prod_{k=1}^{n}P_k  \|_{\mathrm{TV}} \leq (1-|G|\varepsilon_*)^{I^{(m_*)}((k+1) m_*)-1} \leq 2 (1-|G|\varepsilon_*)^{I^{(m_*)}((k+1) m_*)}.
$$
Proposition \ref{propcondSnU} and Lemma \ref{mstarblolem} yield that 
\begin{equation}
     \label{proofbdtvexp}
     \begin{aligned}
  \|\PP(S_{n}=\cdot)- U\|_{\mathrm{TV}} &\leq 2(1-|G|\varepsilon_*)^{C_1 (1-\alpha)^{m_*} (k+1)}+2e^{-C_2 (1-\alpha)^{m_*} (k+1)} \\
  &\leq 2(1-|G|\varepsilon_*)^{\frac{C_1 (1-\alpha)^{m_*} n}{m_*}}+2e^{-\frac{C_2(1-\alpha)^{m_*} n}{m_*}},
\end{aligned}
\end{equation}
where $C_1$ and $C_2$ are the positive constants in Lemma \ref{mstarblolem}. Since $0<(1-\alpha)^{m_*}\leq1$, multiplying the right-hand side by a constant depending only on $G$ and $\mu$ covers the remaining values of $n$, which proves the proposition.
\end{proof}

\subsection{Proof of Proposition \ref{propclassf}: Spectral methods}
\label{specsec}
The convergence of time-inhomogeneous chains that admit an invariant measure have been studied by Saloff-Coste and Z\'u\~niga \cite{MR2340874} via spectral methods, more precisely, singular values techniques. Their results will be used in the proof of Proposition \ref{propclassf}. It is also worth mentioning that they further developed the singular values techniques in \cite{MR2519527}, while the companion paper \cite{MR2789587} discussed Nash and log-Sobolev inequalities techniques.

 For a transition matrix $K=(K(x,y))_{x\in G, y\in G}$ with stationary distribution $U$, we denote by $1=\sigma_0(K)\geq \sigma_1(K) \geq \sigma_2(K)\geq \dots$ the singular values of $K$ arranged in non-increasing order.

\begin{proof}[Proof of Proposition \ref{propclassf}]
Recall $(P_k)_{k\in [n]}:=(P_{k-1,k})_{k\in [n]}$ defined in (\ref{defPkell}).
There are two types of $P_k$, depending on whether $k\in \mathscr{I}_n$: it is either $P_{\mu}$ or $P^{(g)}$ defined in (\ref{defPgmatrix}) for some $g\in \Gamma$. Notice that $P^{(g)}(P^{(g)})^T$ is the identity matrix, and in particular, $\sigma_1(P^{(g)})=1$. On the other hand, the matrix $P_{\mu}$ is also normal since $\mu$ is symmetric, and thus, $\sigma_1(P_{\mu})=\lambda_*$. Consequently,  \cite[Theorem 3.5]{MR2340874} shows that (recall the distance $\chi(\cdot,\cdot)$ defined in (\ref{defchidis}))
$$
\chi(\delta_{e_G}\prod_{k=1}^nP_k, U) \leq \sqrt{|G|-1} \prod_{j=1}^n \sigma_1\left(P_j\right)=\sqrt{|G|-1} \lambda_*^{I(n)}.
$$
Using this inequality, we deduce from Proposition \ref{propcondSnU} and Proposition \ref{Inestprop} that
  \begin{equation}
      \label{inebdTVclasssym}
      \begin{aligned}
    \|\PP(S_{n}=\cdot)- U\|_{\mathrm{TV}} &\leq \frac{\sqrt{|G|-1}}{2}  \EE \lambda_*^{I(n)} \leq \frac{\sqrt{|G|-1}}{2} \left(\lambda_*^{\frac{(1-\alpha)n}{8}}+\PP\left(I(n) \leq  \frac{(1-\alpha)n}{8}\right)\right) \\
    &\leq   \frac{\sqrt{|G|-1}}{2}\left(\lambda_*^{\frac{(1-\alpha)n}{8}}+ 2e^{-\frac{(1-\alpha)n}{18}}\right).
\end{aligned}
  \end{equation}
  We now prove (\ref{ineAbelmix}) for $C=20$. Note that $\lambda_*=1-\gamma_*\leq e^{-\gamma_*}$. If
    $$
n\geq \frac{20}{1-\alpha} \log \left(\frac{|G|}{\varepsilon}\right) \frac{1}{\gamma_*}-1 \geq \frac{18}{1-\alpha} \log \left(\frac{|G|}{\varepsilon}\right) \frac{1}{\gamma_*},
    $$
    where we used that $|G|\geq 2$ and $2\log 2>1$, then by (\ref{inebdTVclasssym}) and that $1/\gamma_*\geq1$, one has
$$
\begin{aligned}
   \|\PP(S_{n}=\cdot)- U\|_{\mathrm{TV}}
   &\leq \frac{\sqrt{|G|-1}}{2}\left(\left(\frac{\varepsilon}{|G|}\right)^{9/4}
   +2\left(\frac{\varepsilon}{|G|}\right)^{1/\gamma_*}\right) \\
   &\leq \frac{3\sqrt{|G|-1}}{2|G|}\varepsilon
   \leq \frac34\varepsilon<\varepsilon.
\end{aligned}
$$
Here we used $\varepsilon/|G|<1$ and $2\sqrt{|G|-1}\leq |G|$.
\end{proof}

\subsection{Proofs of Theorem \ref{thmconj} and Proposition \ref{proplazy}: Evolving sets}
\label{secevosets}

The evolving set process is an auxiliary process taking values in the subsets of the state space, which was introduced by Morris and Peres \cite{MR2198701}. The evolving sets have been used to prove some sharp bounds on mixing times of (time-homogeneous) Markov chains in terms of isoperimetric properties of the state space. This technique has also been applied to dynamical settings, see e.g. \cite{gu2024random, MR4841477, MR3689964, MR4087484, MR4246022}. 

Fix $n\geq 1$, recall the transition probabilities $(P_{k,\ell})_{0\leq k \leq \ell \leq n}$ and $(P_j)_{j \in [n]}$ on $G$ given by (\ref{defPkell}) and (\ref{defPj}) where $G$ does not need to be finite, and each $P_j$ is either $P_{\mu}$ or $P^{(g)}$ for some $g \in \Gamma$. Given $(P_j)_{j\in [n]}$, we define a time-inhomogeneous Markov chain $(W_j)_{0 \leq j \leq n}$ on subsets of $G$ as follows: 
\begin{itemize}
    \item Let $(U_j)_{j \in [n]}$ be i.i.d. random variables uniformly distributed in $(0,1)$.
    \item For $j=0,1,\dots,n-1$, if $W_j=W \subset G$, then 
    $$
W_{j+1}:=\{y\in G: \sum_{x\in W} P_{j+1}(x,y)\geq U_{j+1}\}.
    $$
\end{itemize} 
The chain $(W_j)_{0 \leq j \leq n}$ is called an evolving set process. We denote by $\mathbf{P}$ the law of $(W_j)_{0 \leq j \leq n}$ conditionally on $\sigma(\F_n, (g_j)_{j \in [n]\backslash \mathscr{I}_n})$, and write $\mathbf{P}_{W}$ if we further assume that $W_0=W$.

\begin{lemma}
\label{compleevolset}
    The complement $(W_j^c)_{0\leq j \leq n}$ of the evolving set process is also an evolving set process with the same transition probabilities. \end{lemma}
 \begin{proof}
Let $\mathbf{1}$ be the all-ones vector on $G$. Then $\mathbf{1}$ is an invariant measure for both $P_{\mu}$ and $P^{(g)}$ $(g \in G)$. In particular,
for any $j\in \{0,1,\dots, n-1\}$, the measure $\mathbf{1}$ is invariant under $P_{j+1}$, and thus,
    $$\sum_{x\in W_{j}}P_{j+1}(x,y)=1-\sum_{x\in W_{j}^c}P_{j+1}(x,y).$$ 
 Then, by definition, 
 $$ W_{j+1}^c =\{y \in G: \sum_{x\in W_j} P_{j+1}(x,y)< U_{j+1}\}  =\{y \in G: \sum_{x\in W_{j}^c}P_{j+1}(x,y)\geq 1-U_{j+1}\} \quad \text{a.s.} $$
It remains to note that $(1-U_j)_{j \in [n]}$ are i.i.d. random variables uniformly distributed in $(0,1)$.
\end{proof}

When $G$ is finite, recall that $U$ denotes the uniform measure on $G$. For any subset $W$ of $G$, we write 
$$
W^{\#}:=\begin{cases}
    W & \text{if } U(W) \leq \frac{1}{2} , \\ 
W^c & \text{otherwise,}
\end{cases}
$$
Also recall the $\ell^2$-distance $\chi(\cdot,\cdot)$ defined in (\ref{defchidis}). The following lemma relates $\chi(P_{0,n}(x,\cdot), U)$ to the evolving set process. 
\begin{lemma}
\label{distevolset}
(i). If $W_0$ is finite, then under $\mathbf{P}$, the sequence $(|W_j|)_{0 \leq j \leq n}$ is a martingale with respect to the filtration generated by $(U_j)_{j \in [n]}$, and for any $0\leq k \leq  \ell\leq n$ and $x,y \in G$, one has 
    $$
P_{k,\ell}(x,y)=\mathbf{P}(y\in W_{\ell} \mid W_k=\{x\}).
    $$
 (ii). Assume that $G$ is finite, then for any $0\leq k \leq  \ell\leq n$ and $x\in G$, one has 
    $$
\chi(P_{k,\ell}(x,\cdot),U) \leq |G| \mathbf{E}\left(\sqrt{U(W_{\ell}^{\#})}\mid W_k=\{x\}\right).
    $$
\end{lemma}
\begin{proof}
    The proof of Part (i) is similar to that of \cite[Lemma 2.1]{MR3689964} (with $S_t=W_{k+t}$, $\pi^{(t)}(\cdot)\equiv \mathbf{1}$, $V_t\equiv G$ and $\pi^{(t)}(\cdot,\cdot)= P_{k+t+1}(\cdot,\cdot)$ in the notation there) and we omit the proof details here.

    Given Part (i), the proof of Part (ii) follows the same lines as that of \cite[Equation (24)]{MR2198701} (with the invariant measure $\pi=U$ in the notation there).
\end{proof}

For the rest of this subsection, assume that $G$ is finite. In view of Lemma \ref{distevolset}, it is natural to study the decay of $\mathbf{E}_{\{x\}}\sqrt{U(W_{n}^{\#})}$ as $n\to \infty$. We shall adapt the proof strategy used in \cite{MR2198701} and introduce the following notations: For $j\in [n]$, we let 
   $$
\widehat{K}_j(W, A)=\frac{|A|}{|W|} \mathbf{P}(W_j = A| W_{j-1}=W),
$$ 
where $W,A$ are non-empty subsets of $G$. By Lemma \ref{distevolset} (i), one has $\sum_A \widehat{K}_j(W, A)=1$, and in particular, $(\widehat{K}_j)_{j\in [n]}$ are transition kernels on sets. For any $0\leq k \leq  \ell\leq n$, by induction on $\ell$, one has, 
\begin{equation}
    \label{doobtranind}
    \widehat{\mathbf{P}}(W_{\ell}=A \mid W_k=W)=\frac{|A|\mathbf{P}(W_{\ell}=A \mid W_k=W)}{|W|},
\end{equation}
where we write $\widehat{\mathbf{P}}$ for the probability under which the chain $(W_j)_{0 \leq j \leq n}$ has transition kernels $(\widehat{K}_{j})_{j \in [n]}$ (similarly, $\widehat{\mathbf{E}}$ below denotes the corresponding expectation). In particular, each $W_j$ is a.s. non-empty under $\widehat{\mathbf{P}}_W$. Again, we emphasize that $\widehat{\mathbf{P}}$ is a conditional probability given $\F_n$ and $(g_j)_{j \in [n]\backslash \mathscr{I}_n}$. For $W \subset G$, we define 
\begin{equation}
    \label{defWmu}
    W_{\mu}:=\{y\in G: \sum_{x\in W} P_{\mu}(x,y)\geq \tilde{U}\},
\end{equation}
where $\tilde{U}$ is a uniform random variable in $(0,1)$. Note that 
$$
K_{\mu}(W,A):= \mathbf{P}(W_{\mu}=A), \quad \text{for } A \subset G,
$$
is the transition kernel for the $j$-th step of the evolving set process if $P_j=P_{\mu}$. When $W$ is non-empty, we write
\begin{equation}
    \label{defpsiW}
    \psi_{\mu}(W):=1-\mathbf{E} \left(\sqrt{\frac{|W_{\mu}|}{|W|}} \right)=1-\frac{\sum_{A: A\subset G}\sqrt{|A|}K_{\mu}(W,A)}{\sqrt{|W|}}.
\end{equation}
When $G$ is finite, define the root profile $\psi_{\mu}(r)$ for $r\geq 1/|G|$ by 
\begin{equation}
    \label{defpsir}
    \psi_{\mu}(r):=\inf \{\psi_{\mu}(W): 0<U(W) \leq r\},\quad r \in\left[\frac{1}{|G|}, \frac{1}{2} \right]; \quad \psi_{\mu}(r):=\psi_{\mu}\left(\frac{1}{2}\right), \quad r> \frac{1}{2}.
\end{equation}

We define $W_{\check{\mu}}$, $K_{\check{\mu}}$, $\psi_{\check{\mu}}(W)$ and $\psi_{\check{\mu}}(r)$ in the same way, with $\mu$ replaced by $\check{\mu}(g):=\mu(g^{-1})$. Thus $P_{\check{\mu}}=P_{\mu}^{T}$.

Note that the root profile $\psi_{\mu}(r)$ is decreasing in $r$. The following lemma explains the role of the condition $\langle\Gamma\cdot\Gamma^{-1}\rangle=G$ in the evolving set method: it is equivalent to positivity of the root profile. Its proof will be given later.
\begin{lemma}
\label{psiposilem}
    Assume that $G$ is finite and $P_{\mu}$ is irreducible
and aperiodic. Then, 
$$
\psi_{\mu}\left(\frac{1}{2}\right) >0 \Longleftrightarrow  \langle \Gamma \cdot \Gamma^{-1}  \rangle=G.
$$
\end{lemma}

Proposition \ref{propevopsi} below implies that the distance between $P_{0,n}(e_G,\cdot)$ and $U$ is small if there are sufficiently many (not necessarily consecutive) free steps.

\begin{proposition}
\label{propevopsi}
Under Assumption \ref{mainassum}, if $\langle \Gamma \cdot \Gamma^{-1}  \rangle=G$, then for any $0\leq k \leq  \ell\leq n$ and $x\in G$ and $\varepsilon>0$, 
    $$
\chi^2(P_{k,\ell}(x,\cdot),U) \leq \varepsilon \quad \text { if } |\mathscr{I}_n \cap \{k+1,k+2,\dots,\ell\}|  \geq \int_{4/|G|}^{4 / \varepsilon} \frac{d u}{u \psi_{\mu}(u)}.
$$
The integral is taken to be zero when its upper limit is below its lower limit.
\end{proposition}
\begin{proof}
     When $\varepsilon\geq |G|-1$, the assertion follows from $\chi^2(\nu,U)\leq |G|-1$ for every probability measure $\nu$ on $G$. We may therefore assume that $0<\varepsilon<|G|-1$. For $j\in \{0,1,\dots,n\}$, we write $Z_j:=\sqrt{U(W_j^{\#})}/U(W_j)$ with the convention that $Z_j=0$ if $|W_j|=0$.  By (\ref{doobtranind}) and Lemma \ref{distevolset} (ii), one has 
\begin{equation}
    \label{inechiPnZn}
    \begin{aligned}
          \chi(P_{k,\ell}(x,\cdot),U)&\leq |G| \mathbf{E}\left(\sqrt{U(W_{\ell}^{\#})} \mid W_k=\{x\}\right)= \mathbf{E}\left(|W_{\ell}|\frac{\sqrt{U(W_{\ell}^{\#})}}{U(W_{\ell})}\mid W_k=\{x\}\right) \\
          &=\widehat{\mathbf{E}} ( Z_{\ell} \mid  W_k=\{x\}). 
    \end{aligned}   
\end{equation}
We write $I(k,\ell)=|\mathscr{I}_n \cap \{k+1,k+2,\dots,\ell\}|$, and let $j_1<j_2<\dots<j_{I(k,\ell)}$ be the isolated vertices in $\{k+1,k+2,\dots,\ell\}$. We write $j_0:=k$. Observe that if $j\notin \mathscr{I}_n$, then $P_j=P^{(g)}$ for some deterministic $g\in G$ and $\widehat{K}_j(W,W\cdot g)=\mathbf{P}(W_j = W\cdot g| W_{j-1}=W)=1$, and in particular, for each $m \in [I(k,\ell)]$, the two random variables $ W_{j_{m-1}}$ and $W_{j_{m}-1}$ generate the same $\sigma$-algebra, and the sizes $|W_j|$ are the same for $j=j_{m-1},j_{m-1}+1,\dots,j_{m}-1$ (so are $Z_j$'s). Similarly, $|W_j|$ are the same for $j=j_{I(k,\ell)},j_{I(k,\ell)}+1,\dots,\ell$. Therefore, for any $m \in [I(k,\ell)]$, if $W^{\#}_{j_{m}-1}$ is non-empty, then by the definition of $\widehat{K}$, one has
\begin{equation}
    \label{ineZj}
    \begin{aligned}
    \widehat{\mathbf{E}}\left(\frac{Z_{j_{m}}}{Z_{j_{m-1}}} \mid  W_{j_{m-1}}\right) &= \mathbf{E}\left(\frac{|W_{j_{m}}|}{|W_{j_{m}-1}|}\frac{Z_{j_{m}}}{Z_{j_{m}-1}} \mid W_{j_{m}-1}\right) \\
    &= \mathbf{E}\left(\sqrt{\frac{|W^{\#}_{j_{m}}|}{|W^{\#}_{j_{m}-1}|}} \mid W_{j_{m}-1}\right)\leq 1-\psi_{\mu}(U(W^{\#}_{j_{m}-1})).
\end{aligned}
\end{equation}
Note that the last inequality directly follows from the definition (\ref{defpsir}) when $U(W_{j_{m}-1})\leq 1/2$; when $W_{j_{m}-1}=W$ with $U(W)> 1/2$, the last inequality holds since by Lemma \ref{compleevolset}, one has
$$
\mathbf{E}\left(\sqrt{\frac{|W^{\#}_{j_{m}}|}{|W^{\#}_{j_{m}-1}|}} \mid W_{j_{m}-1}=W\right)\leq \mathbf{E}\left(\sqrt{\frac{|W^{c}_{j_{m}}|}{|W^c|}} \mid W_{j_{m}-1}^c=W^c\right)\leq 1-\psi_{\mu}(U(W^{c})).
$$
Observe that $1-\psi_{\mu}(r)$ is non-decreasing in $r$ and that $U(W^{\#}_{j_{m}-1})\leq Z_{j_{m}-1}^{-2}$. Therefore, (\ref{ineZj}) shows that for any $m \in [I(k,\ell)]$,
\begin{equation}
    \label{ZjellhatEkeyine}
    \widehat{\mathbf{E}}(Z_{j_{m}}\mid  W_{j_{m-1}}) \leq Z_{j_{m-1}}(1-\psi_{\mu}(Z_{j_{m}-1}^{-2})).
\end{equation}
The inequality also holds when $W^{\#}_{j_{m}-1}$ is empty since $\emptyset$ and $G$ are two absorbing states for the evolving set process. By the argument of \cite[Lemma 11 (iii)]{MR2198701}, which applies to every positive target value below the initial value,
$$
\widehat{\mathbf{E}} ( Z_{\ell} \mid  W_k=\{x\}) \leq \sqrt{\varepsilon}, \text { if } I(k,\ell) \geq \int_{4/|G|}^{4 / \varepsilon} \frac{d u}{u \psi_{\mu}(u)},
$$
which completes the proof by (\ref{inechiPnZn}).
\end{proof}

For the proof of Proposition \ref{proplazy}, we shall consider the time-reversal of $(P_j)_{j\in [n]}$, i.e., 
$$
\bar{P}_j(x,y):=P_{n+1-j}(y,x)=P_{n-j,n+1-j}(y,x), \quad j\in [n], x,y \in G.
$$
Note that each $\bar{P}_j$ is a transition kernel since $P_{n+1-j}$ is either $P_{\mu}$ (in which case $P_{\mu}(y,x)=\mu(y^{-1}\cdot x)$) or $P^{(g)}$ for some $g \in G$ (in which case $P^{(g)}(y,x)$ equals $1$ when $x=y\cdot g$, and equals $0$ otherwise). One can easily check by definition that for any $0\leq k \leq \ell \leq n$, 
$$
P_{k,\ell}(x,y)=\bar{P}_{n-\ell,n-k}(y,x)
$$
where $\bar{P}_{n-k,n-k}(x,y)=\delta_{x,y}$ and
$$
\bar{P}_{n-\ell,n-k}:=\bar{P}_{n-\ell+1}\bar{P}_{n-\ell+2}\cdots \bar{P}_{n-k}, \quad \text{ for } k<\ell.
$$
Now observe that for any subset $A \subset G$, since $\sum_{x\in A}\mu(x^{-1}\cdot y)+\sum_{x\in A^c}\mu(x^{-1}\cdot y)=1$,
$$
\begin{aligned}
    P_{\mu}(A,A^c)&=\sum_{y\in A^c}\sum _{x\in A}\mu(x^{-1}\cdot y)=\sum_{y\in A^c}(\sum _{z\in G}\mu(y^{-1}\cdot z)-\sum _{x\in A^c}\mu(x^{-1}\cdot y)) \\
    &=\sum_{y\in A^c, z\in A}\mu(y^{-1}\cdot z)=P_{\mu}(A^c,A).
\end{aligned}
$$
Therefore, $P_{\mu}$ and $P_{\check{\mu}}$ have the same isoperimetric profile $\Phi$. Their root profiles need not coincide. Proposition \ref{propevopsi} also holds for $(\bar{P}_{k,\ell})_{0\leq k\leq \ell \leq n}$, with $\psi_{\mu}$ replaced by $\psi_{\check{\mu}}$ and $\mathscr{I}_n$ replaced by $\bar{\mathscr{I}}_n:=\{j \in [n]: n+1-j \in \mathscr{I}_n\}$. Indeed, the free steps of the reversed chain have kernel $P_{\check{\mu}}$, and its other steps are right translations by elements of $\Gamma^{-1}$. The group condition for $\check{\mu}$ follows from Lemma \ref{sizeincconmu} below.

\begin{proof}[Proof of Proposition \ref{proplazy}]
We may assume that $\Phi(1/2)>0$, since otherwise the upper bound is infinite. By \cite[Lemma 3]{MR2198701}, for $\nu=\mu$ and $\nu=\check{\mu}$, one has
\begin{equation}
    \label{psiPhiine}
    \psi_{\nu}(W) \geq \frac{\mu_0^2\Phi^2(W)}{2(1-\mu_0)^2},
    \qquad \psi_{\nu}(r)\geq \frac{\mu_0^2\Phi^2(r)}{2(1-\mu_0)^2}.
\end{equation}
Set
$$
A_{\varepsilon}:=\frac{(1-\mu_0)^2}{\mu_0^2}\int_{4/|G|}^{8/\varepsilon}\frac{du}{u\Phi^2(u)}.
$$
Since $\Phi\leq1$, $\mu_0\leq1/2$ and $|G|\geq2$, we have
$$
A_{\varepsilon}\geq\log\left(\frac{2|G|}{\varepsilon}\right)>1,
\qquad
\int_{4/|G|}^{8/\varepsilon}\frac{du}{u\psi_{\nu}(u)}\leq2A_{\varepsilon},
\quad \nu\in\{\mu,\check{\mu}\}.
$$
Assume that $I(n)\geq5A_{\varepsilon}$. Let $m$ be the $\lceil2A_{\varepsilon}\rceil$-th isolated vertex in $\mathscr{I}_n$. Then $m<n$ and
$$
|\mathscr{I}_n\cap[m]|\geq2A_{\varepsilon},
\qquad
|\bar{\mathscr{I}}_n\cap[n-m]|
=|\mathscr{I}_n\cap([n]\backslash[m])|\geq2A_{\varepsilon},
$$
where we used $5A_{\varepsilon}\geq1+4A_{\varepsilon}$. Applying Proposition \ref{propevopsi} to the forward and reversed chains gives, for any $x,y\in G$,
$$
\chi^2(P_{0,m}(x,\cdot),U)\leq\frac{\varepsilon}{2},
\qquad
\chi^2(\bar P_{0,n-m}(y,\cdot),U)\leq\frac{\varepsilon}{2}.
$$
The Cauchy--Schwarz inequality now yields
$$
\begin{aligned}
\big||G|P_{0,n}(x,y)-1\big|
&=\left|\sum_{z\in G}\frac{1}{|G|}
(|G|P_{0,m}(x,z)-1)(|G|P_{m,n}(z,y)-1)\right|\\
&\leq\chi(P_{0,m}(x,\cdot),U)
\chi(\bar P_{0,n-m}(y,\cdot),U)\leq\frac{\varepsilon}{2}.
\end{aligned}
$$
Consequently, for any $y\in G$,
$$
\big||G|\PP(S_n=y)-1\big|
\leq\frac{\varepsilon}{2}+|G|\PP(I(n)<5A_{\varepsilon}).
$$
If
$$
n\geq\frac{40A_{\varepsilon}}{1-\alpha}
\geq\frac{40}{1-\alpha}\log\left(\frac{2|G|}{\varepsilon}\right),
$$
then $5A_{\varepsilon}\leq(1-\alpha)n/8$. Proposition \ref{Inestprop} gives
$$
\PP(I(n)<5A_{\varepsilon})
\leq\PP\left(I(n)\leq\frac{(1-\alpha)n}{8}\right)
\leq2\left(\frac{\varepsilon}{2|G|}\right)^{20/9}
<\frac{\varepsilon}{2|G|}.
$$
Thus $d_{\infty}(n)\leq\varepsilon$ for all such $n$, and
$$
t_{\infty}^{(\alpha)}(\varepsilon)
\leq1+\frac{40A_{\varepsilon}}{1-\alpha}
\leq\frac{41(1-\mu_0)^2}{(1-\alpha)\mu_0^2}
\int_{4/|G|}^{8/\varepsilon}\frac{du}{u\Phi^2(u)}.
$$
\end{proof}

To prove Theorem \ref{thmconj}, we shall need the following auxiliary lemma, which will imply Lemma \ref{psiposilem}. Recall that under Assumption \ref{mainassum}, there exists a positive integer $m_*$ such that $P_{\mu}^{m_*}(x,y)>0$ for all $x\in G$ and $y\in G$, and in particular, the set $\Gamma$ generates $G$.

\begin{lemma}
\label{sizeincconmu}
Under Assumption \ref{mainassum}, if $ \langle \Gamma \cdot \Gamma^{-1}  \rangle=\langle \Gamma^{-1} \cdot \Gamma  \rangle$, then for any non-empty subset $A \subset G$ with $|A \cdot \Gamma|=|A|$, one has $A=G$. In particular, 
$$\langle \Gamma \cdot \Gamma^{-1}  \rangle=\langle \Gamma^{-1} \cdot \Gamma  \rangle \Longleftrightarrow \langle \Gamma \cdot \Gamma^{-1}  \rangle=G \Longleftrightarrow \langle \Gamma^{-1} \cdot \Gamma  \rangle =G.$$
\end{lemma}
\begin{proof}
  Assume that $ \langle \Gamma \cdot \Gamma^{-1}  \rangle=\langle \Gamma^{-1} \cdot \Gamma  \rangle$. We argue by contradiction. Suppose there exists a subset $A$ such that  $|A \cdot \Gamma|=|A|$ and $0<|A|<|G|$. Then for any $x,y\in \Gamma$, we have $A\cdot x = A \cdot y$, which implies that $ A \cdot \Gamma \cdot \Gamma^{-1} =A$. We choose $a_1 \in A$ (note that $A$ is non-empty). Then, $e_G\in a_1^{-1}\cdot A$ and $ a_1^{-1}\cdot A \cdot \Gamma \cdot \Gamma^{-1} =a_1^{-1}\cdot A$,
    and in particular, $ \langle \Gamma \cdot \Gamma^{-1}  \rangle \subset a_1^{-1}\cdot A$. 
    
    We now show that $\langle \Gamma \cdot \Gamma^{-1}  \rangle$ is a proper normal subgroup. It is proper since $|\langle \Gamma \cdot \Gamma^{-1}  \rangle| \leq |a_1^{-1}\cdot A| < |G|$. Now, for any $x\in \Gamma$, 
    \begin{equation}
        \label{GaGinsubset}
        x^{-1}\cdot \langle \Gamma \cdot \Gamma^{-1} \rangle \cdot x \subset \langle \Gamma^{-1} \cdot \Gamma  \rangle =\langle \Gamma \cdot \Gamma^{-1} \rangle, 
    \end{equation}
    and similarly, 
 \begin{equation}
        \label{GinGasubset}
 x\cdot \langle \Gamma \cdot \Gamma^{-1} \rangle \cdot x^{-1} =x\cdot \langle \Gamma^{-1} \cdot \Gamma  \rangle \cdot x^{-1}  \subset \langle \Gamma \cdot \Gamma^{-1} \rangle,
\end{equation}
    or equivalently, $ \langle \Gamma \cdot \Gamma^{-1} \rangle  \subset x^{-1}\cdot \langle \Gamma \cdot \Gamma^{-1} \rangle \cdot x$. Since $\Gamma$ generates $G$, we see that $\langle \Gamma \cdot \Gamma^{-1}  \rangle$ is a proper normal subgroup. Fix $x \in \Gamma$, we have $ \Gamma=\Gamma \cdot x^{-1} \cdot x \subset \langle \Gamma \cdot \Gamma^{-1}  \rangle \cdot x$,
    which implies that for any $x_1,x_2,\dots,x_m\in \Gamma$ where $m\geq 1$,
\begin{equation}
\label{apericoset}
  x_1 \cdot x_2 \cdot \dots \cdot x_m \in \langle \Gamma \cdot \Gamma^{-1}  \rangle \cdot x^m,
\end{equation}   
where we used the normality of $\langle \Gamma \cdot \Gamma^{-1}  \rangle$ to get
$$\langle \Gamma \cdot \Gamma^{-1}  \rangle \cdot x \cdot \langle \Gamma \cdot \Gamma^{-1}  \rangle  = \langle \Gamma \cdot \Gamma^{-1}  \rangle \cdot x \cdot x^{-1}  \cdot \langle \Gamma \cdot \Gamma^{-1}  \rangle \cdot x=\langle \Gamma \cdot \Gamma^{-1}  \rangle \cdot x .$$
However, (\ref{apericoset}) shows that for any $m\geq 1$, 
$$
|\{y\in G: P_{\mu}^m(e_G,y)>0\}| \leq |\langle\Gamma \cdot \Gamma^{-1}  \rangle \cdot x^m|<|G|,
$$
which contradicts the existence of $m_*$. Now note that $|\langle \Gamma \cdot \Gamma^{-1}  \rangle \cdot \Gamma|=|\langle \Gamma \cdot \Gamma^{-1} \rangle|$ since
$$
\langle \Gamma \cdot \Gamma^{-1}  \rangle \cdot \Gamma \cdot \Gamma^{-1}=\langle \Gamma \cdot \Gamma^{-1}  \rangle.
$$
Therefore, $\langle \Gamma \cdot \Gamma^{-1}  \rangle=G$ if $\langle \Gamma \cdot \Gamma^{-1}  \rangle=\langle \Gamma^{-1} \cdot \Gamma  \rangle$. The equivalence $\langle \Gamma \cdot \Gamma^{-1}  \rangle=G \Longleftrightarrow \langle \Gamma^{-1} \cdot \Gamma  \rangle =G$ is obvious in view of (\ref{GaGinsubset}) and (\ref{GinGasubset}), in which case, one has $\langle \Gamma \cdot \Gamma^{-1}  \rangle=\langle \Gamma^{-1} \cdot \Gamma  \rangle$.
\end{proof}
 
\begin{proof}[Proof of Lemma \ref{psiposilem}]
Let $W \subset G$ be a nonempty proper subset. Recall the random set $W_{\mu}$ given in (\ref{defWmu}). By Jensen's inequality and Lemma \ref{distevolset}, 
$$
\psi_{\mu}(W)=1-\mathbf{E} \left(\sqrt{\frac{|W_{\mu}|}{|W|}} \right) \geq 1-\sqrt{\mathbf{E} \left(\frac{|W_{\mu}|}{|W|}\right)}=0,
$$
moreover, $\psi_{\mu}(W)=0$ if and only if $|W_{\mu}|=|W|$ a.s.-$\mathbf{P}$. Observe that $W_{\mu}$ is decreasing in $\tilde{U}$ (in terms of the set inclusion). It is easy to see that the maximum set and the minimal set of $W_{\mu}$ are, respectively, given by 
$$
W_{\mu,\text{max}}=W\cdot \Gamma, \quad \text{and}\quad W_{\mu,\text{min}}=\{y \in G: y \cdot \Gamma^{-1}\subset W\}.
$$
Therefore, $\psi_{\mu}(W)=0$ if and only if $W_{\mu,\text{max}}=W_{\mu,\text{min}}$, or equivalently, $W \cdot \Gamma \cdot \Gamma^{-1}=W$, which is impossible if $\langle \Gamma \cdot \Gamma^{-1}  \rangle=G$ by Lemma \ref{sizeincconmu}. On the other hand, since 
$$
\langle \Gamma \cdot \Gamma^{-1}  \rangle \cdot \Gamma \cdot \Gamma^{-1}=\langle \Gamma \cdot \Gamma^{-1} \rangle, \text{ and thus, } (\langle \Gamma \cdot \Gamma^{-1}  \rangle)^c \cdot \Gamma \cdot \Gamma^{-1}=(\langle \Gamma \cdot \Gamma^{-1}  \rangle)^c.
$$
Therefore, if $\langle \Gamma \cdot \Gamma^{-1}  \rangle \neq G$, then $\psi_{\mu}((\langle \Gamma \cdot \Gamma^{-1}\rangle)^{\#})=0$, and thus, $\psi_{\mu}(1/2)=0$.
\end{proof}

\begin{proof}[Proof of Theorem \ref{thmconj}]
By Lemmas \ref{psiposilem} and \ref{sizeincconmu},
$$
\eta:=\min\{\psi_{\mu}(1/2),\psi_{\check{\mu}}(1/2)\}>0.
$$
For $\varepsilon\in(0,1)$ and $\nu\in\{\mu,\check{\mu}\}$, we have
$$
\int_{4/|G|}^{8/\varepsilon}\frac{du}{u\psi_{\nu}(u)}
\leq\frac{1}{\eta}\log\left(\frac{2|G|}{\varepsilon}\right).
$$
Since $\eta\leq1$ and $\log(2|G|/\varepsilon)>1$, the splitting argument in the proof of Proposition \ref{proplazy} shows that
$$
\big||G|P_{0,n}(e_G,y)-1\big|\leq\frac{\varepsilon}{2},\quad y\in G,
\qquad\text{if }\quad
I(n)\geq\frac{3}{\eta}\log\left(\frac{2|G|}{\varepsilon}\right).
$$
Indeed, the last threshold is at least one plus the sum of the two integrals. Therefore, if
$$
n\geq\frac{24}{(1-\alpha)\eta}\log\left(\frac{2|G|}{\varepsilon}\right),
$$
Proposition \ref{Inestprop} gives
$$
d_{\infty}(n)
\leq\frac{\varepsilon}{2}+|G|\PP\left(I(n)<\frac{(1-\alpha)n}{8}\right)
\leq\frac{\varepsilon}{2}+2|G|e^{-(1-\alpha)n/18}.
$$
For $(1-\alpha)n>24\log(2|G|)/\eta$, choose
$\varepsilon=2|G|e^{-\eta(1-\alpha)n/24}$.
Since $\eta/24\leq1/24<1/18$, we obtain
$d_{\infty}(n)\leq3|G|e^{-\eta(1-\alpha)n/24}$.
For the remaining values of $n$, use $d_{\infty}(n)\leq|G|-1$.
Both cases are covered by
$$
d_{\infty}(n)\leq3|G|^2e^{-\eta(1-\alpha)n/24},\qquad n\geq1.
$$
This proves the first assertion.

In view of Lemma \ref{sizeincconmu}, it remains to show that in cases \textbf{(i),(ii),(iii)}, one has $ \langle \Gamma \cdot \Gamma^{-1}  \rangle=\langle \Gamma^{-1} \cdot \Gamma  \rangle$. \textbf{(i)}. By definition, $\Gamma \cdot \Gamma^{-1}=\Gamma^{-1} \cdot \Gamma$ if $\Gamma$ is symmetric. \textbf{(ii)}. Assume that $\Gamma$ is a union of conjugacy classes of $G$. In particular, for any $x,y\in G$, one has $\mu(x\cdot y)>0$ if and only if $\mu(y \cdot x)=\mu(x^{-1}\cdot x\cdot y\cdot x)>0$. Now observe that an element $z\in G$ is in $\Gamma \cdot \Gamma^{-1}$, resp. $\Gamma^{-1} \cdot \Gamma$, if and only if 
$$
\sum_{x\in G} \mu(x) \mu(z^{-1}\cdot x) >0, \quad \text{resp.}\quad \sum_{x\in G} \mu(x) \mu(x\cdot z^{-1}) >0.
$$
Therefore, one has $\Gamma \cdot \Gamma^{-1}  =\Gamma^{-1} \cdot \Gamma$. If $G$ is abelian, then each conjugacy class is a singleton set. \textbf{(iii)}. If $e_G\in \Gamma$, then it is easy to see that $\langle \Gamma \cdot \Gamma^{-1}  \rangle=\langle \Gamma \cup \Gamma^{-1}  \rangle=\langle \Gamma^{-1} \cdot \Gamma  \rangle$.
\end{proof}

\begin{remark}
Theorem \ref{thmconj} also admits a proof by singular values. The symmetric matrix $P_{\mu}P_{\mu}^{T}$ is the transition matrix of the walk with step distribution $\mu*\check{\mu}$, whose support is $\Gamma\cdot\Gamma^{-1}$. Thus
$$
\langle\Gamma\cdot\Gamma^{-1}\rangle=G
\quad\Longleftrightarrow\quad s:=\sigma_1(P_{\mu})<1.
$$
As in the proof of Proposition \ref{propclassf}, each free step contracts the $\ell^2(U)$ norm on functions of mean zero by at most $s$, whereas right translations preserve this norm. Hence
$$
\chi(P_{0,n}(e_G,\cdot),U)\leq\sqrt{|G|-1}\,s^{I(n)},
$$
with the convention $0^0=1$. Using $\|\nu/U-1\|_{\infty}\leq\sqrt{|G|}\chi(\nu,U)$ and Proposition \ref{Inestprop}, we obtain
$$
\begin{aligned}
d_{\infty}(n)&\leq |G|\EE s^{I(n)}\\
&\leq |G|\left(e^{-(1-s)(1-\alpha)n/8}
+2e^{-(1-\alpha)n/18}\right).
\end{aligned}
$$
This gives another proof of the exponential bound, including when $s=0$.
\end{remark}
\subsection{Proof of Theorem \ref{phasetrancycle}: Long-range jumps speed up mixing}
\label{secabel}

This section is devoted to the proof of Theorem \ref{phasetrancycle}. In particular, for $\alpha\geq 1/2$, we shall prove the following upper bound for $t^{(\alpha)}_{\mathrm{mix}}(\varepsilon)$. 
\begin{proposition}
\label{propspeedup}
  In the setting of Theorem \ref{phasetrancycle}, we further assume that $\alpha\in [1/2,1)$. Then for any $L\geq 3$, one has
 $$
  t^{(1/2)}_{\mathrm{mix}}(\varepsilon) \leq \frac{C_1 L^2}{\log L}, \text{ if }\alpha =1/2; \quad  t^{(\alpha)}_{\mathrm{mix}}(\varepsilon) \leq C_2 L^{\frac{1}{\alpha}}, \text{ if }\alpha >1/2,
  $$
  where $C_1=C_1(\varepsilon)$ and $C_2=C_2(\alpha,\varepsilon)$ are positive constants not depending on $L$.
\end{proposition}

The proof of Proposition \ref{propspeedup} will be given later. Taking Proposition \ref{propspeedup} for granted, we prove Theorem \ref{phasetrancycle}.

\begin{proof}[Proof of Theorem \ref{phasetrancycle}]
 For any fixed $\varepsilon \in (0,1)$, since $t^{(0)}_{\mathrm{mix}}(\varepsilon) \to \infty$ as $L\to \infty$,  Proposition \ref{propabelcompare} (i) then shows that for all large $L$,
  $$t^{(\alpha)}_{\mathrm{mix}}(\varepsilon) \leq  \frac{8}{1-\alpha} t^{(0)}_{\mathrm{mix}}\left(\frac{\varepsilon}{2}\right)+1.$$
 Then the upper bound in (i) is a direct consequence of the well-known result that $t^{(0)}_{\mathrm{mix}}(\varepsilon/2)=O(L^2)$, see e.g. \cite[Theorem 2, Chapter 3C]{MR964069}. When $\alpha \in [1/2,1)$, the upper bounds for $t^{(\alpha)}_{\mathrm{mix}}(\varepsilon)$ stated in (ii) and (iii) follow from Proposition \ref{propspeedup}. It remains to prove the lower bounds for $t^{(\alpha)}_{\mathrm{mix}}(\varepsilon)$ in (i), (ii) and (iii).  For any $\varepsilon \in (0,1)$, choose $\tilde{\varepsilon}$ such that $\varepsilon < \tilde{\varepsilon} <1$. By the definition of total variation distance, one has (we take $A=\{x \in \ZZ_L: d(x,0) \geq (1-\tilde{\varepsilon})L/2\}$ where $d(\cdot,\cdot)$ is the graph distance on $\ZZ_L$)
  \begin{equation}
    \label{lowbdmixcycle}
     \|\PP(S_{n}=\cdot)- U\|_{\mathrm{TV}} \geq U(A)-\PP(S_n \in A) \geq \tilde{\varepsilon}-\frac{1}{L}-\PP\left(|\widehat{S}_n| \geq \frac{(1-\tilde{\varepsilon})L}{2}\right),
  \end{equation}
 where $\widehat{S}$ is the elephant random walk on $\ZZ$ defined in Section \ref{srrwonRd} with the same parameter $\alpha$ and step distribution uniform on $\{-1,1\}$ (note that $\widehat{S}_n \mod L$ and $S_n$ have the same distribution). If we take $n= \lceil K L^2 \rceil$ $(\alpha<1/2)$, $n= \lceil K L^2/\log L \rceil$ $(\alpha=1/2)$ or $n= \lceil K L^{1/\alpha} \rceil$ $(\alpha>1/2)$ for some constant $K>0$, then by the Markov inequality and the following estimates of the second moment of $\widehat{S}_n$ (see \cite{schutz2004elephants})
 $$ \EE \widehat{S}_n^2 \sim \frac{n}{1-2\alpha}, \text{ if }\alpha<\frac{1}{2}; \quad \EE \widehat{S}_n^2 \sim n \log n, \text{ if }\alpha=\frac{1}{2}; \quad \EE \widehat{S}_n^2 \sim \frac{n^{2\alpha}}{(2\alpha-1)\Gamma(2\alpha)}, \text{ if }\alpha>\frac{1}{2},$$
 we see that the right-hand side of (\ref{lowbdmixcycle}) is larger than $\varepsilon$ for all large $L$ if $K$ is sufficiently small, which proves the desired lower bounds. 
 
 In fact, the constants $c_1,c_2$ and $c_3$ can be chosen to go to infinity as $\varepsilon\to 0$. Write
 $$
 a_L=\begin{cases}
 L^2, &\alpha<1/2,\\
 L^2/\log L, &\alpha=1/2,\\
 L^{1/\alpha}, &\alpha>1/2.
 \end{cases}
 $$
 Since the mean of the function $x\mapsto e^{2\pi\mathrm{i}x/L}$  under $U$ is zero, one has
 \begin{equation}
 \label{charlowbdcycle}
 \|\PP(S_n=\cdot)-U\|_{\mathrm{TV}}
 \geq \frac12\left|\EE e^{2\pi\mathrm{i}\widehat{S}_n/L}\right|.
 \end{equation}
 For $n=\lceil t a_L\rceil$ with fixed $t>0$, (\ref{CLTERW}) gives
 $$
 \lim_{L\to\infty}\EE e^{2\pi\mathrm{i}\widehat{S}_n/L}
 =\begin{cases}
 \exp\{-2\pi^2t/(1-2\alpha)\}, &\alpha<1/2,\\
 \exp\{-4\pi^2t\}, &\alpha=1/2.
 \end{cases}
 $$
 Here the factor $2$ in the limiting variance when $\alpha=1/2$ comes from $\log n/\log L\to2$.

 Suppose now that $\alpha>1/2$, and let $\phi(u)=\EE e^{\mathrm{i}uW}$. By \cite[Corollary 2.11]{MR5079601}, with $p=(1+\alpha)/2$ and $q=1/2$ in their notation, $\EE e^{r|W|}<\infty$ for every $r>0$. This implies that $\phi$ cannot vanish on a nonempty open interval. Indeed, if it did, then at an interior point $u_0$ we would have $\phi^{(j)}(u_0)=\EE[(\mathrm{i}W)^j e^{\mathrm{i}u_0W}]=0$ for every integer $j\geq0$. Expanding $e^{-\mathrm{i}u_0W}$ and taking expectations term by term would give
 $$
 1=\EE\left[e^{\mathrm{i}u_0W}e^{-\mathrm{i}u_0W}\right]
 =\sum_{j=0}^{\infty}\frac{(-u_0)^j}{j!}\phi^{(j)}(u_0)=0,
 $$
 a contradiction. The interchange is justified by
 $\sum_{j\geq0}|u_0|^j\EE|W|^j/j!=\EE e^{|u_0||W|}<\infty$.
 Thus, for any $M>0$, we can choose a real number $t\geq M$ such that $\phi(2\pi t^\alpha)\neq0$. By (\ref{superalphaW}) and (\ref{charlowbdcycle}),
 $$
 \liminf_{L\to\infty}\|\PP(S_{\lceil t a_L\rceil}=\cdot)-U\|_{\mathrm{TV}}
 \geq\frac12|\phi(2\pi t^\alpha)|>0.
 $$
 The Gaussian limits above give the same conclusion for $\alpha\leq1/2$, for every $t>0$. Our definition of mixing time requires the distance to stay below $\varepsilon$ at all later times. Thus, if the distance at time $\lceil t a_L\rceil$ is larger than $\varepsilon$, then $t^{(\alpha)}_{\mathrm{mix}}(\varepsilon)>\lceil t a_L\rceil$. Consequently, for every $M>0$ and all sufficiently small $\varepsilon>0$,
 $$
 \liminf_{L\to\infty}\frac{t^{(\alpha)}_{\mathrm{mix}}(\varepsilon)}{a_L}\geq M.
 $$
 This proves the assertion about $c_1,c_2,c_3$. Together with the upper bound $t^{(\alpha)}_{\mathrm{mix}}(1/4)=O(a_L)$, it also shows that there is no cutoff.
\end{proof}

The main idea of the proof for Proposition \ref{propspeedup} can be briefly described as follows. Recall from (\ref{Sndecomabel}) that for any $n\geq 1$, we can write  
\begin{equation}
    \label{Sndecomcycle}
    S_n=\sum_{j=1}^n \left|\CalC_{j, n}\right| g_j, \ \text{ with }\ \PP(g_1=1)=\PP(g_1=-1)=\frac{1}{2}
\end{equation}
where $(g_j)_{j\geq 1}$ are i.i.d. random variables independent of $\F_n$. The same sum, viewed in $\ZZ$, gives the integer-valued walk $\widehat{S}_n$. Conditionally on $(\left|\CalC_{j, n}\right|)_{1\leq j\leq n}$, the random variable $S_n$ is the sum of independent steps (random variables) $\left|\CalC_{j, n}\right| g_j$, $j=1,2,\dots,n$. In the proof of Proposition \ref{propabelcompare}, we simply use the free steps (i.e., $g_j$ with $\left|\CalC_{j, n}\right|=1$). To prove Proposition \ref{propspeedup}, we shall also need long-range steps (i.e., $\left|\CalC_{j, n}\right| g_j$ with large $\left|\CalC_{j, n}\right|$). Conditionally on the forest, a cluster of size $r$ contributes a factor $\cos^2(ru)$ to the squared modulus of the characteristic function at frequency $u$. When $r|u|\leq\pi/2$, this factor is at most $e^{-r^2u^2}$. We use sufficiently many clusters at suitable scales to obtain stronger decay than that provided by the isolated vertices alone.

More precisely, for any probability measure $\nu$ on $\ZZ_L$, we have
\begin{equation}
    \label{upbdlemZL}
    \|\nu-U\|_{\mathrm{TV}}^2 \leq \frac{1}{4} \chi^2(\nu,U) =\frac{1}{4} \sum_{k=1}^{L-1}\left|\sum_{m=0}^{L-1} e^{-\mathrm{i} 2 km\pi/L} \nu(m) \right|^2,
\end{equation}
which follows from the upper bound lemma \cite{MR626813} (see also \cite[Lemma 1, Chapter 3B]{MR964069}) and the fact that the set of non-trivial irreducible representations of $(\ZZ_L, +)$ 
is given by $\{\chi_k\}_{1 \leq k\leq L-1}$ where $\chi_k(m):=e^{\mathrm{i} 2 km\pi/L}$ for $m\in \ZZ_L$, see e.g. \cite[Example 4.4.10]{MR2867444}. Using (\ref{Sndecomcycle}) and (\ref{upbdlemZL}), one has
\begin{equation}
    \label{UplemSncos}
    \begin{aligned}
    \|\PP(S_n=\cdot)-U\|_{\mathrm{TV}}^2 &\leq \frac{1}{4} \sum_{k=1 }^{L-1}|\EE e^{\frac{-\mathrm{i} 2\pi k \widehat{S}_n}{L}}|^2=\frac{1}{2} \sum_{k=1 }^{(L-1)/2}|\EE e^{\frac{\mathrm{i} \pi k \widehat{S}_n}{L}}|^2 \\
    &\leq \frac{1}{2} \sum_{k=1 }^{(L-1)/2}\left|\EE \prod_{j=1}^n \cos\left(\frac{\pi k |\CalC_{j,n}|}{L}\right)\right|^2 \leq \frac{1}{2} \sum_{k=1 }^{(L-1)/2}\EE \prod_{j=1}^n \cos^2\left(\frac{\pi k |\CalC_{j,n}|}{L}\right),
\end{aligned}
\end{equation}
where the first equality uses that $L$ is odd and $\widehat{S}_n\equiv n\pmod 2$. Indeed, the latter implies that $|\EE e^{\mathrm{i}u\widehat{S}_n}|$ is $\pi$-periodic. Multiplication by $2$ permutes the nonzero residues modulo $L$, and opposite frequencies have the same modulus. On the other hand, we have shown in Proposition \ref{cluster_sizewindowestlem} that for each $k=1,2,\dots, \lfloor L/40 \rfloor$, with high probability, there is ``a sufficient number" of clusters $\CalC_{j,n}$ such that
\begin{equation}
     \label{lowupbdcjn}
    \frac{L}{96k} \leq |\CalC_{j,n}|< \frac{L}{2k},
\end{equation}
and for $k>\lfloor L/40 \rfloor$, we use Proposition \ref{Inestprop} to show that with high probability, there is ``a sufficient number" of isolated vertices in $\F_n$, which all satisfy (\ref{lowupbdcjn}). This would imply that the right-hand side of (\ref{UplemSncos}) is small.

\begin{corollary}
    \label{upcorbdexpkey}

(i). If $\alpha \in (1/2,1)$ and $L\geq 120$, then there exists a positive constant $C=C(\alpha)$ such that for any $k \in [1, (L-1)/2]$ and any $n\geq H L^{\frac{1}{\alpha}}$ with $H\geq 1$, one has
  $$
  \EE \exp\left(-\frac{\pi^2 k^2}{L^2} \sum_{j=1}^n |\CalC_{j,n}|^2 \mathds{1}_{\{|\CalC_{j,n}|\leq \frac{L}{2k}\}}\right) \leq 5 e^{-C H k^{\frac{1}{\alpha}}}.
  $$
  (ii). If $\alpha=1/2$, then there exist positive constants $L_0$, $C$ and $\hat{C}$ such that for any $L\geq L_0$ and any $k \in [1, (L-1)/2]$ and any $n\geq H L^{2}/\log L$ with $H\geq 1$, one has
$$
  \EE \exp\left(-\frac{\pi^2 k^2}{L^2} \sum_{j=1}^n |\CalC_{j,n}|^2 \mathds{1}_{\{|\CalC_{j,n}|\leq \frac{L}{2k}\}}\right) \leq \hat{C} e^{-  C H k }.
  $$
\end{corollary}
\begin{proof}
Let $C_1$ and $C_2$ be as in Proposition \ref{cluster_sizewindowestlem}. \\
(i). Assume that $\alpha \in (1/2,1)$ and $L\geq 120$. Then by Proposition \ref{cluster_sizewindowestlem}, for any $k \in [1, \lfloor L/40 \rfloor]$ and $n\geq H L^{\frac{1}{\alpha}}$ with $H\geq 1$, we have 
\begin{equation}
    \label{bdcyclewindow}
    \begin{aligned}
    &\quad\ \PP\left(\frac{\pi^2 k^2}{L^2} \sum_{j=1}^n |\CalC_{j,n}|^2 \mathds{1}_{\{|\CalC_{j,n}|\leq \frac{L}{2k}\}}\leq \frac{\pi^2 C_2 H k^{\frac{1}{\alpha}} }{96^2 }\right) \\
    &\leq \PP\left( \left|\left\{1\leq j \leq n: \frac{L}{96k} \leq |\CalC_{j,n}|<\frac{L}{2k}\right\}\right| \leq  \frac{C_2 k^{\frac{1}{\alpha}}n}{L^{\frac{1}{\alpha}}}\right) \leq  4e^{- C_1 H k^{\frac{1}{\alpha}}}.
\end{aligned}
\end{equation}
In particular, letting $C_3:=\min\{\pi^2 C_2/96^2,C_1\}$, one has
$$
\EE \exp\left(-\frac{\pi^2 k^2}{L^2} \sum_{j=1}^n |\CalC_{j,n}|^2 \mathds{1}_{\{|\CalC_{j,n}|\leq \frac{L}{2k}\}}\right) \leq 4 e^{-C_1 Hk^{\frac{1}{\alpha}}} + e^{-C_3 H k^{\frac{1}{\alpha}}}
\leq 5 e^{-C_3 H k^{\frac{1}{\alpha}}},
$$
which proves the desired inequality for $k \in [1, \lfloor L/40 \rfloor]$. For $k \in ( \lfloor L/40 \rfloor,(L-1)/2]$, observe that 
$$
\frac{\pi^2 k^2}{L^2} \sum_{j=1}^n |\CalC_{j,n}|^2 \mathds{1}_{\{|\CalC_{j,n}|\leq \frac{L}{2k}\}} \geq \frac{\pi^2 k^2}{L^2} \sum_{j=1}^n |\CalC_{j,n}|^2 \mathds{1}_{\{|\CalC_{j,n}|=1\}} \geq\frac{\pi^2}{1600} I(n).
$$
Then using Proposition \ref{Inestprop}, we have 
\begin{equation}
    \label{bdcyclewindowlargek}
\EE \exp\left(-\frac{\pi^2 k^2}{L^2} \sum_{j=1}^n |\CalC_{j,n}|^2 \mathds{1}_{\{|\CalC_{j,n}|\leq \frac{L}{2k}\}}\right) \leq e^{-C_4n}+2e^{-(1-\alpha)n/18}\leq 3e^{-C_4n},
\end{equation}
where $C_4:=\pi^2(1-\alpha)/(1600\cdot8)<(1-\alpha)/18$.
It remains to note that $n\geq HL^{\frac{1}{\alpha}} \geq Hk^{\frac{1}{\alpha}}$ and set $C:=\min\{C_3,C_4\}$.

(ii). Assume that $\alpha=1/2$ and $L\geq L_0$ where $L_0 \geq 400$ is sufficiently large and will be determined later (in particular, $L/\sqrt{\log L}\geq 120$), and  $n\geq H L^{2}/\log L$ with $H\geq 1$. For each integer $1\leq k\leq\log L$, we shall now consider not just one window of cluster sizes, but a sequence of windows: For each $m \geq 1$ such that
\begin{equation}
    \label{windowm}
\log L \leq k \cdot 96^{m} \leq \lfloor \frac{L}{40} \rfloor,\quad \text{which holds when } \frac{\log (\frac{\log L}{k})}{ \log 96} \leq m \leq  \frac{\log ( \frac{1}{k} \lfloor \frac{L}{40} \rfloor   )}{\log 96},
\end{equation}
similar to (\ref{bdcyclewindow}), Proposition \ref{cluster_sizewindowestlem} (with $k$ there replaced by $k \cdot 96^m$) implies that
$$ 
 \PP\left( \sum_{j=1}^n |\CalC_{j,n}|^2 \mathds{1}_{\{ \frac{L}{96^{m+1} \cdot k} \leq |\CalC_{j,n}|<\frac{L}{ 96^{m} \cdot k}\}} \leq \frac{C_2 n }{4\cdot 96^2}\right)   \leq 4 \exp\left(-\frac{C_1 96^{2m}  H k^2   }{4 \log L} \right).
$$
Observe that as $L\to \infty$, 
$$
\frac{1}{\log 96} \left(\log ( \frac{1}{k} \lfloor \frac{L}{40} \rfloor   ) - \log (\frac{\log L}{k})\right) = \frac{1}{\log 96} \log \left( \frac{\lfloor \frac{L}{40} \rfloor}{\log L} \right) \sim \frac{\log L}{\log 96}.
$$
Since $k\leq\log L$, the lower endpoint of the interval in (\ref{windowm}) is nonnegative. Requiring $m\geq1$ therefore removes at most one integer. Thus, there exists $L_0\geq400$ such that for $L\geq L_0$, the number of admissible $m$'s is at least $\log L/(2\log96)$. The corresponding windows are disjoint and are contained in $(0,L/(96k)]$. By the union bound, one has
$$
\begin{aligned}
   &\quad \ \PP\left(\frac{\pi^2 k^2}{L^2} \sum_{j=1}^n |\CalC_{j,n}|^2 \mathds{1}_{\{|\CalC_{j,n}|\leq \frac{L}{2k}\}}\leq  \frac{C_2 \pi^2 H  k^2  }{8\cdot 96^2 \cdot \log 96}  \right) \\
   &\leq  \PP\left( \sum_{j=1}^n |\CalC_{j,n}|^2 \mathds{1}_{\{|\CalC_{j,n}|\leq \frac{L}{96 k}\}}\leq   \frac{C_2 n \log L}{8\cdot 96^2 \cdot \log 96}  \right)   \\ 
   &\leq 4 \sum_{m}  \exp\left(-\frac{C_1 96^{2m}  H k^2   }{4 \log L} \right) \leq 4 \sum_{m}  \exp\left(-\frac{C_1 96^{m}  H k   }{4} \right)\leq C_6 e^{-24 C_1  H k   }.
\end{aligned}
$$
Here the summation in the third line is taken over all $m$'s satisfying (\ref{windowm}) (and in particular, $k\cdot 96^m \geq \log L$), and the last inequality holds for some positive constant $C_6$ because 
$$
\sum_{m}  \exp\left(-\frac{C_1 96^{m}  H k   }{4} \right)\leq  \sum_{m=1}^{\infty} e^{- 24 C_1   m H k   },  \quad \text{(note that } 96^m \geq 96 m \text{ for all } m\geq 1).
$$
Therefore, 
$$
\EE \exp\left(-\frac{\pi^2 k^2}{L^2} \sum_{j=1}^n |\CalC_{j,n}|^2 \mathds{1}_{\{|\CalC_{j,n}|\leq \frac{L}{2k}\}}\right) \leq \exp\left(- \frac{C_2 \pi^2 H  k^2  }{8\cdot 96^2 \cdot \log 96}\right) +   C_6 e^{-24 C_1  H k },
$$
which proves the desired inequality for $k\leq\log L$. For $\log L<k\leq(L-1)/2$, the isolated vertices give
 $$
 \begin{aligned}
 &\EE\exp\left(-\frac{\pi^2k^2}{L^2}\sum_{j=1}^n|\CalC_{j,n}|^2
 \mathds{1}_{\{|\CalC_{j,n}|\leq L/(2k)\}}\right)\\
 &\qquad\leq\EE\exp\left(-\frac{\pi^2k^2}{L^2}I(n)\right)
 \leq\exp\left(-\frac{\pi^2k^2n}{16L^2}\right)+2e^{-n/36}\\
 &\qquad\leq e^{-\pi^2Hk/16}+2e^{-Hk/36}.
 \end{aligned}
 $$
 Here we used Proposition \ref{Inestprop}, $k^2/\log L>k$, and $L^2/\log L\geq k$. This completes the proof.
\end{proof}

\begin{proof}[Proof of Proposition \ref{propspeedup}]
We only prove the case when $\alpha \in (1/2,1)$; the case when $\alpha=1/2$ can be proved similarly using Corollary \ref{upcorbdexpkey} (ii) instead of Corollary \ref{upcorbdexpkey} (i). We first assume that $L\geq 120$. Using (\ref{UplemSncos}) and that $\cos x \leq e^{-x^2/2}$ for $x\in [0,\pi/2]$, if $n\geq  H L^{\frac{1}{\alpha}}$ with $H\geq 1$, one has 
$$
\begin{aligned}
  \|\PP(S_n=\cdot)-U\|_{\mathrm{TV}}^2&\leq \frac{1}{2} \sum_{k=1 }^{(L-1)/2} \EE \exp\left(-\frac{\pi^2 k^2}{L^2} \sum_{j=1}^n |\CalC_{j,n}|^2 \mathds{1}_{\{|\CalC_{j,n}|\leq \frac{L}{2k}\}}\right) \\
  &\leq \frac52 \sum_{k=1 }^{\infty} e^{-C H k^{\frac{1}{\alpha}}} \leq \frac52\sum_{k=1 }^{\infty} e^{-C H k} = \frac{5 e^{-C H}}{2(1-e^{-C H})},
\end{aligned}
$$
where we used Corollary \ref{upcorbdexpkey} (i) in the second inequality. For any $\varepsilon\in(0,1)$, when $H = \frac{2 \log(3/\varepsilon)}{C}  +1$, the last term is smaller than $\varepsilon^2$, and in particular, $t^{(\alpha)}_{\mathrm{mix}}(\varepsilon) \leq H L^{\frac{1}{\alpha}}+1\leq (H+1)L^{\frac{1}{\alpha}}$. For each odd $3\leq L<120$, we can find a real number $H_L>0$ such that $t^{(\alpha)}_{\mathrm{mix}}(\varepsilon) \leq H_L L^{\frac{1}{\alpha}}$. Then $C=C(\alpha,\varepsilon):=\max\{H+1,H_3,H_5,\dots,H_{119}\}$ satisfies the requirement. In the case $\alpha=1/2$, the finitely many odd $L<L_0$ are handled in the same way.
\end{proof}

\subsection{Proof of Theorem \ref{slowhycu}: Counting effective steps}
\label{secerwpoly}

Let $J$ be the set of positive odd integers. Since every element of $\ZZ_2^d$ has order at most $2$, only clusters of odd size contribute to the sum in (\ref{Sndecomabel}). Thus, the number of effective steps is $N_J(n)$. Write
$$
p_\alpha:=\sum_{k\in J}\theta_k
=\frac{1-\alpha}{2}\ {}_2F_1(1,\alpha^{-1};\alpha^{-1}+1;\frac12),
$$
where the last equality follows from Corollary \ref{corNJn}. Throughout this section, $\alpha\in(0,1)$ is fixed. For $x=(x(1),\ldots,x(d))\in\{0,1\}^d$, let $W(x):=\sum_{k=1}^d x(k)$ be the number of $1$'s in $x$. The following lemma compares $W(S_n)$ with the corresponding quantity for the lazy simple random walk at a deterministic time.

\begin{lemma}
\label{lemLhycudomiexp}
Let $\tilde S$ be the lazy simple random walk on $G=(\ZZ_2^d,+)$ with step distribution $\mu$ given in Theorem \ref{slowhycu}, starting from $\tilde S_0=e_G$, and let $S$ be its step-reinforced version with reinforcement parameter $\alpha\in(0,1)$. Then for any $y\geq0$, $n\geq1$ and integer $r\geq0$,
$$
\PP(W(S_n)\geq y)
\leq\PP(W(\tilde S_r)\geq y)+\PP(N_J(n)>r).
$$
\end{lemma}

\begin{proof}
Fix $n\geq1$ and $r\geq0$. Using (\ref{Sndecomabel}), we can couple $S_n$ and $\tilde S_r$ as follows:
\begin{equation}
\label{coupSlazy}
S_n:=\sum_{i=1}^{N_J(n)}g_i,
\qquad
\tilde S_r:=\sum_{i=1}^r g_i,
\end{equation}
where $(g_i)_{i\geq1}$ are i.i.d. with law $\mu$ and are independent of $\F_n$. To generate these variables, let $(u_i^{(d)})_{i\geq1}$ be i.i.d. uniform on $\{1,\ldots,d\}$, and let $(h_i)_{i\geq1}$ be independent Bernoulli variables with parameter $1/2$, independent of the coordinate choices and the forest. Set $g_i=h_i e_{u_i^{(d)}}$. Let
$$
C_u:=\{u_i^{(d)}:1\leq i\leq N_J(n)\},
\qquad
\tilde C_u:=\{u_i^{(d)}:1\leq i\leq r\}.
$$
Conditionally on $\F_n$ and $(u_i^{(d)})_{i\geq1}$, the coordinates of each sum in (\ref{coupSlazy}) are independent. Each coordinate that has been chosen is uniform on $\{0,1\}$, since it is the sum modulo $2$ of at least one independent Bernoulli variable with parameter $1/2$; the other coordinates are zero. Hence the two conditional laws of the numbers of $1$'s are $\operatorname{Bin}(|C_u|,1/2)$ and $\operatorname{Bin}(|\tilde C_u|,1/2)$, respectively. On $\{N_J(n)\leq r\}$, we have $C_u\subset\tilde C_u$, and the latter binomial law stochastically dominates the former. Taking expectations gives
$$
\PP(W(S_n)\geq y,N_J(n)\leq r)
\leq\PP(W(\tilde S_r)\geq y,N_J(n)\leq r).
$$
Adding $\PP(N_J(n)>r)$ proves the lemma.
\end{proof}

\begin{proof}[Proof of Theorem \ref{slowhycu}]
Fix $\varepsilon\in(0,1)$ and set
$$
t_d:=\frac{d\log d}{2p_\alpha}.
$$
We first prove the upper bound. Let $m=t^{(0)}_{\mathrm{mix}}(\varepsilon/2)$. By (\ref{coupSlazy}) and the contraction of total variation distance under a Markov kernel,
$$
\|\PP(S_n=\cdot)-U\|_{\mathrm{TV}} \leq\EE\|\delta_{e_G}P_\mu^{N_J(n)}-U\|_{\mathrm{TV}} \leq\frac{\varepsilon}{2}+\PP(N_J(n)<m).
$$
The usual lazy simple random walk satisfies $m\leq(d\log d)/2+C(\varepsilon/2)d$, see e.g. \cite{MR3726904}. Let
$$
n_0:=\left\lceil\frac{m+d}{p_\alpha}\right\rceil.
$$
For $d>2$ and every $n\geq n_0$, (\ref{concentrationNJbound}) gives
$$
\PP(N_J(n)<m)
\leq2\exp\left(-\frac{(p_\alpha n-m-2)^2}{4n}\right)
\leq2\exp\left(-\frac{(d-2)^2}{4n_0}\right).
$$
Indeed, the function $(p_\alpha x-m-2)^2/x$ is increasing for $x\geq(m+2)/p_\alpha$. The last bound tends to zero as $d\to\infty$, since $n_0=O_{\alpha,\varepsilon}(d\log d)$. It follows that the total variation distance is at most $\varepsilon$ for every $n\geq n_0$, for all sufficiently large $d$. This proves
$$
t^{(\alpha)}_{\mathrm{mix}}(\varepsilon)
\leq n_0\leq t_d+\frac{C(\varepsilon/2)+1}{p_\alpha}d+1.
$$

For the lower bound, we also use $U$ for a random variable uniformly distributed on $\ZZ_2^d$. Then $W(U)\sim\operatorname{Bin}(d,1/2)$, so
$$
\EE W(U)=\frac d2,
\qquad \operatorname{Var}(W(U))=\frac d4.
$$
For the lazy simple random walk, we have
$$
\EE W(\tilde S_r)=\frac d2\left(1-\left(1-\frac1d\right)^r\right),
\qquad \operatorname{Var}(W(\tilde S_r))\leq\frac d4,
$$
see e.g. \cite[Proposition 7.14]{MR3726904} (the walk there starts from the all-ones vector). Taking
$$
y:=\frac{\EE W(U)+\EE W(\tilde S_r)}2
=\frac d2\left(1-\frac12\left(1-\frac1d\right)^r\right),
$$
Chebyshev's inequality gives, for $d>1$,
$$
\begin{aligned}
\PP(W(U)\geq y)-\PP(W(\tilde S_r)\geq y)
&\geq1-\frac8d\left(1-\frac1d\right)^{-2r}\\
&\geq1-\frac8d\exp\left(\frac{2r}{d-1}\right),
\end{aligned}
$$
where the last step uses $\log(1-1/d)\geq-1/(d-1)$. Lemma \ref{lemLhycudomiexp} therefore implies
$$
\|\PP(S_n=\cdot)-U\|_{\mathrm{TV}}
\geq1-\frac8d\exp\left(\frac{2r}{d-1}\right)-\PP(N_J(n)>r).
$$
Choose a positive constant $A$ and take
$$
n:=\lfloor t_d-Ad\rfloor,
\qquad r:=\lceil p_\alpha n+d\rceil.
$$
For all sufficiently large $d$, (\ref{concentrationNJbound}) gives
$$
\PP(N_J(n)>r)
\leq2\exp\left(-\frac{(d-2)^2}{4n}\right)=o(1).
$$
Moreover, $r\leq(d\log d)/2+(1-p_\alpha A)d+1$, and hence
$$
\liminf_{d\to\infty}\|\PP(S_n=\cdot)-U\|_{\mathrm{TV}}
\geq1-8e^{2(1-p_\alpha A)}.
$$
Choosing a large $A=A(\alpha,\varepsilon)$ so that the right-hand side is greater than $\varepsilon$, we obtain $t^{(\alpha)}_{\mathrm{mix}}(\varepsilon)>n$ for all sufficiently large $d$. Together with the upper bound, this proves
$$
t^{(\alpha)}_{\mathrm{mix}}(\varepsilon)
=\frac{d\log d}{2p_\alpha}+O_{\alpha,\varepsilon}(d).
$$
\end{proof}

\section{Acknowledgments}

We are grateful to Mark Sellke who suggested an improvement by a logarithmic factor of the upper bound in Theorem \ref{phasetrancycle} for $\alpha=1/2$. Yuval Peres is supported by the National Natural Science Foundation of China under Grant Number W2531011. Shuo Qin is supported by the China Postdoctoral Science Foundation under Grant Number 2025M773086, 2026T190815, and National Natural Science Foundation of China under Grant Number 12271284.

\bibliographystyle{plain}
\bibliography{math_ref}

\end{document}